\documentclass[11pt]{amsart}
\IfFileExists{lmodern.sty}{\usepackage[T1]{fontenc}\usepackage{lmodern}}{}
\usepackage{amsmath,amssymb,mathtools,mathrsfs}
\usepackage[margin=1.08in]{geometry}
\usepackage{microtype,xcolor}
\usepackage[colorlinks=true,linkcolor=blue!40!black,citecolor=blue!40!black,urlcolor=blue!40!black]{hyperref}
\newtheorem{theorem}{Theorem}[section]
\newtheorem{maintheorem}{Theorem}

\newtheorem{proposition}[theorem]{Proposition}
\newtheorem{lemma}[theorem]{Lemma}

\theoremstyle{definition}

\theoremstyle{remark}
\newtheorem{remark}[theorem]{Remark}

\newcommand{\dd}{\,d}
\newcommand{\E}{\mathbb E_\nu}
\newcommand{\KL}{D_{\mathrm{KL}}}
\newcommand{\eps}{\varepsilon}
\newcommand{\R}{\mathbb R}
\newcommand{\Cupper}{\overline{C}_U}
\newcommand{\Clower}{\underline{C}_U}
\newcommand{\bb}{\mathbf b}
\DeclareMathOperator{\Var}{Var}
\DeclareMathOperator{\diver}{div}
\DeclareMathOperator{\Ric}{Ric}

\numberwithin{equation}{section}
\hypersetup{pdftitle={Endpoint capacity at strictly stable minimal boundaries},pdfauthor={Sehong Park},pdfsubject={Endpoint p-capacity, Gaussian boundary layers, logarithmic moments, Schwarzschild deficits}}
\begin{document}
\title[Endpoint capacity at minimal boundaries]{Endpoint capacity at strictly stable minimal boundaries:\\
Gaussian layers and Schwarzschild deficits}
\author{Sehong Park}
\address{Human-Centered Artificial Intelligence Research Institute,
Ewha Womans University, Seoul, Republic of Korea}
\email{pshong21@ewha.ac.kr}
\begin{abstract}
Let $(M^n,\mathfrak g)$, $3\le n\le7$, be a one-ended asymptotically flat manifold
whose boundary $\Sigma$ is a connected, strictly stable, outermost minimal
hypersurface. As $p\downarrow1$, the $p$-capacity of $\Sigma$ converges to its
area, by a theorem of Fogagnolo and Mazzieri. We establish a higher-order
asymptotic expansion governed by a single function on $\Sigma$:
the positive solution $g$ of the reciprocal Jacobi equation
$Jg=g^{-1}$. After a universal term, the first
correction is determined by the average of $\log g$, and the capacitary
potential drops from one to zero across a Gaussian boundary layer of width
$g\sqrt{p-1}$; under the normalized energy measure, the logarithm of the
gradient has limiting variance $\frac12$ plus the variance of $\log g$.
In dimension three, assuming in addition that
$R_{\mathfrak g}+|\mathrm{II}|^2\ge0$ along $\Sigma$, we identify the capacity
deficit of $\Sigma$ relative to the Schwarzschild horizon of the same area
at order $p-1$: it is nonnegative, vanishes exactly when the induced metric of $\Sigma$ is round
and $R_{\mathfrak g}+|\mathrm{II}|^2\equiv0$, and controls the deviation of
$g$ from a constant through a relative-entropy identity.
\end{abstract}
\maketitle

% ===== sec1_intro.tex =====
\section{Introduction}\label{sec:intro}

For $1<p<n$ the $p$-capacity of a compact set $K\subset\R^n$ is the
infimum of $\int_{\R^n}|\nabla u|^p\dd x$ over $u\in C^\infty_c(\R^n)$ with
$u\ge1$ on $K$ \cite{Mazya2011,HeinonenKilpelainenMartio1993}; for $p=2$ it
is, up to normalization, the classical capacity of \cite{PolyaSzego1951}.
For convex bodies its dependence on geometry is the subject of the
Brunn--Minkowski inequality \cite{ColesantiSalani2003}, the Hadamard
variational formula \cite{ColesantiEtAl2015}, and the capacity--area
inequalities of \cite{Xiao2017}.

As $p\downarrow1$, $p$-Laplacian problems approach their
total-variation counterparts, as seen in the convergence of
first eigenvalues to Cheeger constants \cite{KawohlFridman2003}
and of $p$-harmonic functions to functions of least gradient
\cite{Juutinen2005}. For $p$-capacitary potentials $u_p$,
a geometric limit emerges: the logarithmic transform
$-(p-1)\log u_p$ converges to the arrival time of the weak
inverse mean curvature flow of Huisken--Ilmanen
\cite{HuiskenIlmanen2001}. This connection was established
in $\mathbb R^n$ by Moser \cite{Moser2007} and extended
to manifolds under lower sectional curvature bounds
\cite{KotschwarNi2009} or Ricci curvature bounds
\cite{MariRigoliSetti2022}, with suitable conditions at infinity.
More recently, Benatti--Mari--Rigoli--Setti--Xu
\cite{BenattiMariRigoliSettiXu2026} obtained proper solutions
under Ricci or isoperimetric hypotheses through decay estimates
for the $p$-Green kernel.

At the level of capacity, the corresponding limit is the
perimeter of the strictly outward minimizing hull of a bounded
region, both in $\mathbb R^n$
\cite{AgostinianiFogagnoloMazzieri2022} and on complete
manifolds satisfying a Euclidean-type isoperimetric inequality
\cite{FogagnoloMazzieri2022}.

Beyond the leading order much less is known. For balls the first
eigenvalue is the Cheeger constant times $1+(p-1)|\log(p-1)|+O(p-1)$
\cite{BenediktDrabek2013}, as is the explicit one-dimensional eigenvalue
\cite{KajikiyaTakeuchi2025}. Peres--Wang \cite{PeresWang2026} give geometric
conditions under which, at a smooth boundary point of a Euclidean domain
where indicator data vanish, the inward normal derivative of the
$p$-harmonic extension is of order $(p-1)^{-1}$, and others under which it
decays like $e^{-c/(p-1)}$; which regime occurs is not determined locally,
and at the center of the flat bottom face of a cylinder the derivative has the
intermediate order $(p-1)^{-1/2}$.

For horizons, capacity enters through Bray's proof of the Riemannian
Penrose inequality, which rests on a mass--capacity inequality for a
minimal boundary \cite{Bray2001}. Bray--Miao \cite{BrayMiao2008} bounded the
harmonic capacity of a surface in an asymptotically flat three-manifold
with nonnegative scalar curvature by its area and Willmore energy, with
equality only for Schwarzschild; for a minimal boundary the bound is the
capacity of the Schwarzschild horizon of the same area. Xiao \cite{Xiao2016} extended these estimates to the
$p$-capacity for $1<p<3$. Subsequent work has developed
monotonicity formulas for harmonic and $p$-harmonic
potentials \cite{Miao2025,HirschMiaoTam2024}.
For a connected minimal boundary $\Sigma$ with
$H_2(M,\Sigma)=0$ in an asymptotically flat three-manifold with
$R_{\mathfrak g}\ge0$, the $p$-capacity of $\Sigma$ is at most that of the
Schwarzschild horizon of the same area for every $1<p<3$
\cite{Xiao2016,BenattiFogagnoloMazzieri2023}, with equality
only for Schwarzschild \cite{XiaYinZhou2024}.
Agostiniani--Mantegazza--Mazzieri--Oronzio
\cite{AgostinianiMantegazzaMazzieriOronzio2025} obtain the Penrose
inequality by letting $p\downarrow1$ in a mass--$p$-capacity inequality;
as they point out, the equality case is lost in this limit.

In this paper $\Sigma$ is a strictly stable, outermost minimal boundary of
an asymptotically flat manifold. Such a boundary is strictly outward
minimizing (see the proof of Lemma~\ref{lem:log-control}(c)), so by
\cite{FogagnoloMazzieri2022}, applied after smoothly filling in $\Sigma$ (see also
\cite{AgostinianiMantegazzaMazzieriOronzio2025}), the limit above reads
\begin{equation}\label{eq:area-limit}
 \lim_{p\downarrow1}C_M(p)=|\Sigma| ,
\end{equation}
with $C_M(p)$ the capacity \eqref{eq:exterior-capacity}, taken without a
normalizing constant. We ask which geometric information about $\Sigma$ the
correction terms after \eqref{eq:area-limit} contain, and how this
information controls the potential $u_p$ in the layer near $\Sigma$ where it
drops from one to zero.

Both answers are expressed through one function on $\Sigma$. Let $N$ be the
unit normal pointing into $M$, $h$ the induced metric, and
\begin{equation}\label{eq:Jacobi-intro}
 J=-\Delta_h+V,\qquad V=-\bigl(|\mathrm{II}|^2+\Ric_{\mathfrak g}(N,N)\bigr),
\end{equation}
the Jacobi operator, $\Delta_h=\diver_h\nabla$. Strict stability means
$\lambda_1(J)>0$, and then there is a unique smooth positive $g$ on $\Sigma$
with
\begin{equation}\label{eq:nonlinear-J}
 Jg=g^{-1} .
\end{equation}
This equation appears in Smith's construction of black hole initial data
with a horizon of prescribed geometry \cite{Smith2009,Smith2011}, where
existence and uniqueness are proved for $\Sigma\cong S^2$.

\subsection*{Setting and main results}
Let $(M^n,\mathfrak g)$, $3\le n\le7$, be smooth, connected, orientable and
complete, with compact connected boundary $\Sigma$ and one asymptotically
flat end,
\begin{equation}\label{eq:AF}
 \mathfrak g_{ij}=\delta_{ij}+O_2(r^{-\tau}),\qquad \tau>(n-2)/2 .
\end{equation}
Assume that $\Sigma$ is minimal, strictly stable, and outermost, that is, no
closed minimal hypersurface in $M\setminus\Sigma$ bounds a compact region
together with $\Sigma$. Mean
curvatures are taken with respect to the unit normal pointing away from
$\Sigma$. Write
\[
 A=|\Sigma|,\qquad d\nu=A^{-1}dA,\qquad
 \E\phi=\int_\Sigma\phi\dd\nu,\qquad \Var_\nu\phi=\E\phi^2-(\E\phi)^2 .
\]
The capacity and its potential are
\begin{equation}\label{eq:exterior-capacity}
\begin{gathered}
 C_M(p)=\inf\Bigl\{\int_M|\nabla u|^p\dd V:\
 u\in\operatorname{Lip}_c(\overline M),\ u|_\Sigma=1\Bigr\},\qquad 1<p<n,\\
 \Delta_pu_p=0\ \text{in }M,\qquad u_p|_\Sigma=1,\qquad u_p\to0\ \text{at infinity},
\end{gathered}
\end{equation}
with $C_M(p)=\int_M|\nabla u_p|^p\dd V$.
Throughout, $\eps=p-1$, and
\[
 \gamma_+(y)=\sqrt{\tfrac2\pi}\,e^{-y^2/2}\quad(y\ge0),\qquad
 P_+(y)=\int_y^\infty\gamma_+(s)\dd s .
\]

\begin{maintheorem}[Endpoint capacity and boundary layer]\label{thm:global}
Under the preceding assumptions the following hold as $\eps\downarrow0$.
\begin{enumerate}
\item[\textup{(i)}] \textup{(Capacity asymptotics.)} The capacity has the
expansion
\begin{equation}\label{eq:global-capacity-main}
 \log\frac{C_M(1+\eps)}{A}
 =\frac\eps2|\log\eps|
 +\eps\Bigl(-\frac12\log\frac\pi2-\E\log g\Bigr)
 +\kappa_+\eps^{3/2}+O(\eps^2),
\end{equation}
where $\kappa_+=\frac1{3A}\sqrt{\frac2\pi}\,\frac{d^3}{dt^3}\big|_{t=0}|\Sigma_t|$
and $\Sigma_t=\{\exp_x(tg(x)N(x)):x\in\Sigma\}$.
\item[\textup{(ii)}] \textup{(Gaussian layer.)} In the orthogonal collar
coordinates $(t,x)$ whose leaves are the $\Sigma_t$ of \textup{(i)}, so that
$\partial_t=gN$ on $\Sigma$, let
$R_\eps(t,x)=(x,t/\sqrt\eps)$, and let
$d\mu_p=C_M(p)^{-1}|\nabla u_p|^p\dd V$ be the normalized energy measure.
For every sufficiently small fixed $d>0$, the following
estimates hold:
\begin{equation}\label{eq:global-TV-main}
 \left\|
 (R_\eps)_\#\bigl(\mu_{1+\eps}\llcorner\{t<d\}\bigr)
 -\nu\otimes\gamma_+(y)\dd y
 \right\|_{\mathrm{TV}}
 \le C\sqrt{\eps}
\end{equation}
for measures on $\Sigma\times[0,\infty)$,
\begin{equation}\label{eq:global-profile-intro}
 \sup_{x\in\Sigma}\sup_{0\le y\le d/\sqrt\eps}\bigl|u_{1+\eps}(\sqrt\eps\,y,x)-P_+(y)\bigr|
 \le C\sqrt\eps ,
\end{equation}
and, on $\Sigma$,
\begin{equation}\label{eq:global-normal-intro}
 \Bigl|\sqrt\eps\,g\,|\nabla u_{1+\eps}|-\sqrt{\tfrac2\pi}\Bigr|
 \le C\sqrt\eps .
\end{equation}
\item[\textup{(iii)}] \textup{(Logarithmic variance.)} The variance of
$\log|\nabla u_{1+\eps}|$ under $\mu_{1+\eps}$ satisfies
\begin{equation}\label{eq:global-variance-main}
 \lim_{\eps\downarrow0}\Var_{\mu_{1+\eps}}\bigl(\log|\nabla u_{1+\eps}|\bigr)
 =\frac12+\Var_\nu(\log g).
\end{equation}
\end{enumerate}
\end{maintheorem}

\begin{remark}\label{rem:thmA}
The variance in (iii) converges at rate
$O(\sqrt\eps|\log\eps|^2)$. Moreover,
$p\mapsto\log C_M(p)$ is differentiable for $p>1$
sufficiently close to one, with
\[
 \frac{d}{dp}\log C_M(p)
 =\int_M\log|\nabla u_p|\dd\mu_p.
\]
Thus the derivative of $\log C_M$ in the exponent is the mean of the
logarithmic gradient, whose variance is given in \textup{(iii)}.
\end{remark}

In dimension three we compare $\Sigma$ with the Schwarzschild horizon of
the same area, for which $g\equiv r_A:=\sqrt{A/(4\pi)}$. Assume in addition
that $n=3$ and
\[
 \mathcal R:=\tfrac12\bigl(R_{\mathfrak g}+|\mathrm{II}|^2\bigr)\big|_\Sigma\ge0 .
\]
Let $G$ be the positive solution of $(-\Delta_h+K)G=G^{-1}$, $K$ the Gauss
curvature of $h$, which exists since these assumptions force
$-\Delta_h+K>0$, and put $d(h)=\E\log(G/r_A)$. Let $C_{\mathrm{Sch}}(p;A)$ be
the capacity of the horizon in the Schwarzschild exterior with horizon
area $A$.

\begin{maintheorem}[Schwarzschild endpoint deficit]\label{thm:sch-global}
Under these assumptions the limit
\begin{equation}\label{eq:Delta-intro}
 \Delta_M:=\lim_{\eps\downarrow0}\frac1\eps
 \log\frac{C_{\mathrm{Sch}}(1+\eps;A)}{C_M(1+\eps)}
 =\E\log\frac g{r_A}
\end{equation}
exists and $\Delta_M\ge d(h)\ge0$, where $d(h)>0$ exactly when $h$ is not
round and $\Delta_M=d(h)$ exactly when $\mathcal R\equiv0$. Moreover
\begin{equation}\label{eq:Delta-stability-intro}
 \Bigl\|\frac{g^{-2}}{\E g^{-2}}-1\Bigr\|_{L^1(\nu)}\le2\sqrt{\Delta_M},\qquad
 \E|\nabla\log g|^2+\E\mathcal R\le\frac{8\pi}A\,\Delta_M .
\end{equation}
\end{maintheorem}

\begin{remark}
In particular, $\Delta_M=0$ exactly when $h$ is round
and $\mathcal R\equiv0$. Otherwise,
$C_M(p)<C_{\mathrm{Sch}}(p;A)$ for all $p>1$
sufficiently close to one.

The stability estimates in \eqref{eq:Delta-stability-intro}
follow from the identity
\[
 2\Delta_M=\KL(\nu\Vert\rho)-\log(1-\theta),
\]
where
\[
 d\rho=\frac{g^{-2}}{\E g^{-2}}\dd\nu,
 \qquad
 \theta=\frac A{4\pi}
 \bigl(\E|\nabla\log g|^2+\E\mathcal R\bigr).
\]
Pinsker's inequality gives the first estimate, while
the nonnegativity of the two terms on the right gives
the second. The latter also controls the excess variance
$\Var_\nu(\log g)$ in Theorem~\ref{thm:global}(iii)
through the Poincar\'e inequality.

Unlike the global Schwarzschild comparison recalled above,
Theorem~\ref{thm:sch-global} identifies the endpoint deficit
using only the boundary curvature condition, without assuming
$R_{\mathfrak g}\ge0$ throughout $M$.
The rigidity statement concerns the boundary geometry
and does not imply that the exterior is Schwarzschild.
\end{remark}

\subsection*{The selection principle}
A positive lapse $v$ on $\Sigma$ gives a collar swept out by the normal
graphs of $tv$, and two elementary bounds for its capacity: an upper bound
from one-dimensional competitors \eqref{eq:primal-bound} and, when $Jv>0$, a
lower bound from a divergence-free field of unit flux \eqref{eq:dual-bound}.
Their coefficients of $\eps$ differ by half the Jensen gap
$\log\E(vJv)-\E\log(vJv)$ \eqref{eq:gap-Jensen}, so the bounds agree at first
order exactly when $v$ is a constant multiple of $g$. For $v=g$ they agree up
to relative order $\eps^2$, and the degenerate convexity of $|\cdot|^{1+\eps}$
turns this into $O(\sqrt\eps)$-closeness of the gradient of the collar
minimizer to an explicit Gaussian field \eqref{eq:near-min}.

\subsection*{Organization}
Section~\ref{sec:prelim} sets up the local analysis at a minimal
hypersurface: orthogonal collar coordinates adapted to a lapse $v$,
an upper bound for the collar capacity from one-dimensional competitors
and a lower bound from a divergence-free flux field, and a convexity
estimate. Section~\ref{sec:local} selects the lapse $g$, matches the bounds
and identifies $\kappa_+$, gives the half-Gaussian energy law and the
error-function profile, and computes the logarithmic mean and variance.
Section~\ref{sec:global} uses the weak inverse mean curvature flow to show
that $u_p$ is exponentially small outside the collar
(Lemma~\ref{lem:log-control}), so that the capacity, gradient, energy
measure and profile of the global problem agree with those of the collar
problem up to exponentially small errors
(Proposition~\ref{prop:localization}); this proves
Theorem~\ref{thm:global}. Section~\ref{sec:boundary-gap} obtains the
Schwarzschild expansion from Theorem~\ref{thm:global} and proves
Theorem~\ref{thm:sch-global}. Appendix~\ref{app:asymptotics} contains the
Laplace-type expansions, the two-sided layer and the one-dimensional
barriers; Appendix~\ref{app:imcf} shows
that the weak inverse mean curvature flow leaves $\Sigma$ quadratically,
at normal distance $g\sqrt{2\sigma}$ at time $\sigma$.

% ===== sec2_prelim.tex =====
\section{Preliminaries and primal--dual asymptotics}\label{sec:prelim}

Let $(M^n,\mathfrak g)$, $n\ge2$, be a smooth Riemannian manifold,
possibly with boundary, and let $\Sigma\subset M$
be a closed connected embedded minimal
hypersurface, possibly a boundary component of $M$.
Fix a unit normal $N$ pointing into $M$ when $\Sigma$
is a boundary component, and toward the chosen side
when $\Sigma$ is an interior hypersurface. We consider
a one-sided tubular neighborhood of $\Sigma$ in the
direction of $N$.

Strict stability is used in the form
\begin{equation}\label{eq:strict-stability}
 \int_\Sigma vJv\dd A=\int_\Sigma\bigl(|\nabla v|^2+Vv^2\bigr)\dd A
 \ \ge\ \lambda_1(J)\int_\Sigma v^2\dd A\qquad(v\in H^1(\Sigma)),\qquad \lambda_1(J)>0 .
\end{equation}
Throughout, $0<\eps\le1$ and all capacities are taken with the exponent
$1+\eps$; since $\eps$ is fixed in the local analysis, it is suppressed from
the notation of the collar capacities. All geometric data are fixed before
$\eps\downarrow0$, and $C,c$ denote positive constants depending only on the
fixed data.

\subsection{Collar coordinates and the collar capacity}\label{sec:coordinates}
Let $v>0$ be smooth on $\Sigma$. Moving $\Sigma$ along the normal graphs
$\exp_x(tv(x)N(x))$ and removing the tangential part of the velocity by a
$t$-dependent diffeomorphism of $\Sigma$ gives orthogonal coordinates on a
one-sided collar of $\Sigma$. We identify this collar with
$[0,d_1)\times\Sigma$ for some $d_1>0$, so that $\Sigma=\{0\}\times\Sigma$,
$\partial_t=vN$ on $\Sigma$, and
\begin{equation}\label{eq:collar-metric}
 \mathfrak g=\ell(t,x)^2\dd t^2+h_t,\qquad
 dV=\ell\,e^{w}\dd t\dd A,\qquad e^{w(t,x)}=\frac{dA_{h_t}}{dA_h}(x),
\end{equation}
where $h_t$ is a family of metrics on $\Sigma$ with $h_0=h$ and
$\ell(0,\cdot)=v$. All computations below take place on $[0,d]\times\Sigma$
for a small fixed $d<d_1$.
With $\mathbf n=\ell^{-1}\partial_t$ the unit normal of the leaves,
$\partial_th_t=2\ell\,\mathrm{II}_t$, so
$w_t=\frac12\operatorname{tr}_{h_t}\partial_th_t=\ell H_t$, where $H_t$ is the
mean curvature of the leaf $\{t\}\times\Sigma$ with respect to $\mathbf n$.
At $t=0$, $H_0=0$ gives $w=w_t=0$, and the linearization
$\partial_tH_t|_{t=0}=J\ell(0,\cdot)=Jv$ of the mean curvature at a minimal
hypersurface \cite{ColdingMinicozzi2011} gives
$w_{tt}(0,\cdot)=vJv$. Hence, with smooth $b_1,a_3$ on $\Sigma$ and
remainders uniform in $x$,
\begin{equation}\label{eq:collar-expansions}
 \ell=v\bigl(1+b_1t+O(t^2)\bigr),\qquad
 w=\tfrac12\,vJv\,t^2+a_3t^3+O(t^4).
\end{equation}
For an interior two-sided hypersurface the coordinates exist for both signs
of $t$.

Let $U=(0,d)\times\Sigma$ and $\Sigma_d=\{d\}\times\Sigma$, and define the collar
capacity, its minimizer and its energy measure by
\begin{equation}\label{eq:capacity}
\begin{gathered}
 C_U=\inf\Bigl\{\int_U|\nabla u|^{1+\eps}\dd V:\ u\in W^{1,1+\eps}(U),\ u|_\Sigma=1,\ u|_{\Sigma_d}=0\Bigr\}
 =\int_U|\nabla v_\eps|^{1+\eps}\dd V,\\
 d\mu_\eps^U=C_U^{-1}|\nabla v_\eps|^{1+\eps}\dd V .
\end{gathered}
\end{equation}
The minimizer $v_\eps$ exists and is unique, and $0\le v_\eps\le1$
\cite{HeinonenKilpelainenMartio1993}; $v_\eps\in C^{1,\alpha}(\overline U)$
for each fixed $\eps$ \cite{Lieberman1988}, and $v_\eps$ is smooth on
$\{\nabla v_\eps\ne0\}$.

\subsection{Upper and lower bounds}\label{sec:primal}
\emph{The upper bound.} Let $\phi$ be nonincreasing with
$\phi(0)=1$ and $\phi(d)=0$. For $u=\phi(t)$,
$|\nabla u|=|\phi'|/\ell$ and
\begin{align*}
 \int_U|\nabla u|^{1+\eps}\dd V&=\int_0^d|\phi'(t)|^{1+\eps}\Bigl(\int_\Sigma\ell^{-\eps}e^{w}\dd A\Bigr)\dd t,\\
 1=\int_0^d|\phi'|\dd t&\le\Bigl(\int_0^d|\phi'|^{1+\eps}\!\int_\Sigma\ell^{-\eps}e^w\Bigr)^{\frac1{1+\eps}}
 \Bigl(\int_0^d\Bigl(\int_\Sigma\ell^{-\eps}e^w\Bigr)^{-\frac1\eps}\dd t\Bigr)^{\frac\eps{1+\eps}},
\end{align*}
with equality when $|\phi'|\propto(\int_\Sigma\ell^{-\eps}e^w)^{-1/\eps}$.
Hence
\begin{equation}\label{eq:primal-bound}
 C_U\le\Cupper:=\Bigl(\int_0^d\Bigl(\int_\Sigma\ell(t,x)^{-\eps}e^{w(t,x)}\dd A\Bigr)^{-1/\eps}\dd t\Bigr)^{-\eps}.
\end{equation}

\emph{The lower bound.} Suppose $Jv>0$ and put
\begin{equation}\label{eq:flux-field}
 X=\frac{(Jv/v)^{\eps/2}(x)}{\int_\Sigma(Jv/v)^{\eps/2}\dd A}\;\frac{e^{-w(t,x)}}{\ell(t,x)}\,\partial_t .
\end{equation}
Since $dV=\ell e^w\dd t\dd A$ in the coordinates \eqref{eq:collar-metric},
$X$ is a smooth vector field on $\overline U$ with
\begin{equation}\label{eq:flux-divfree}
 \diver X=0,\qquad
 \int_{\{t\}\times\Sigma}\langle X,\mathbf n\rangle\dd A_{h_t}=1\qquad(0\le t\le d),
\end{equation}
so the divergence theorem gives, for every $u\in W^{1,1+\eps}(U)$ with
$u=1$ on $\Sigma$ and $u=0$ on $\Sigma_d$,
\begin{equation}\label{eq:flux-identity}
 -\int_UX\cdot\nabla u\dd V=1 .
\end{equation}
By H\"older's inequality,
$1\le\bigl(\int_U|X|^{(1+\eps)/\eps}\bigr)^{\eps/(1+\eps)}\bigl(\int_U|\nabla u|^{1+\eps}\bigr)^{1/(1+\eps)}$, so
\begin{equation}\label{eq:dual-bound}
 \Clower:=\Bigl(\int_U|X|^{(1+\eps)/\eps}\dd V\Bigr)^{-\eps}\le C_U\le\Cupper .
\end{equation}

\subsection{Asymptotics of the bounds and the comparison field}\label{sec:asymptotics}
Let $v>0$ be smooth; then $\E(vJv)>0$ by \eqref{eq:strict-stability}. The Laplace expansion of the
two integrals in \eqref{eq:primal-bound} and
\eqref{eq:dual-bound}, carried out in Appendix~\ref{app:bounds}, gives
\begin{equation}\label{eq:primal-full}
 \log\frac{\Cupper}A
 =\frac\eps2|\log\eps|
 +\eps\Bigl(\frac12\log\frac{\E(vJv)}{\pi/2}-\E\log v\Bigr)+O(\eps^{3/2}),
\end{equation}
and if $Jv>0$,
\begin{equation}\label{eq:dual-full}
 \log\frac{\Clower}A
 =\frac\eps2|\log\eps|
 +\eps\Bigl(\frac12\E\log\frac{Jv}{v}-\frac12\log\frac\pi2\Bigr)+O(\eps^{3/2}).
\end{equation}
The remainders depend on $v$ and $d$; the displayed coefficients depend
only on $v$.

\emph{The comparison field.}
To describe the minimizer $v_\eps$, we seek an explicit
vector field approximating $-\nabla v_\eps$ whose
$L^{1+\eps}$ energy equals the lower bound $\Clower$.
The dual field $X$ provides such a comparison field:
\begin{equation}\label{eq:b-definition}
 |\bb|^{\eps-1}\bb=\Clower\,X,
 \qquad\text{equivalently}\qquad
 \bb=\Clower^{1/\eps}|X|^{(1-\eps)/\eps}X .
\end{equation}
Thus $\bb$ is a positive multiple of $\partial_t$, our candidate for
$-\nabla v_\eps$. Since $\Clower^{1/\eps}=(\int_U|X|^{(1+\eps)/\eps}\dd V)^{-1}$,
\begin{equation}\label{eq:b-identities}
 \int_U|\bb|^{1+\eps}\dd V=\Clower,\qquad \int_UX\cdot\bb\dd V=1 ,
\end{equation}
and, by \eqref{eq:flux-field},
\begin{equation}\label{eq:b-scalar-general}
 |\bb|=\frac{|X|^{1/\eps}}{\int_U|X|^{(1+\eps)/\eps}\dd V}
 =\frac{\bigl(\int_\Sigma(Jv/v)^{\eps/2}\dd A\bigr)^{-1/\eps}}{\int_U|X|^{(1+\eps)/\eps}\dd V}
 \Bigl(\frac{Jv}v\Bigr)^{1/2}e^{-w/\eps},
\end{equation}
a fixed profile in $x$ times the Gaussian factor $e^{-w/\eps}$.

We conclude this section with a convexity estimate used
to compare the collar minimizer with the comparison field
and, later, the global potential with the collar minimizer.
All norms below are taken in $L^{1+\eps}$ unless otherwise indicated.

\begin{lemma}\label{lem:convexity}
For $0<\eps\le1$ and vectors $a,b$,
\begin{equation}\label{eq:convexity}
 \frac{|a|^{1+\eps}-|b|^{1+\eps}}{1+\eps}
 -|b|^{\eps-1}b\cdot(a-b)
 \ge \frac\eps2 (|a|+|b|)^{\eps-1}|a-b|^2,
\end{equation}
where $|b|^{\eps-1}b:=0$ for $b=0$ and the right-hand side is taken to be zero when $a=b=0$.
\end{lemma}

\begin{proof}
The function $f(z)=|z|^{1+\eps}/(1+\eps)$ satisfies
$\nabla f=|z|^{\eps-1}z$ and
$\nabla^2 f\ge\eps|z|^{\eps-1}I$ away from the origin.
Along $z_s=b+s(a-b)$, since $|z_s|\le|a|+|b|$,
$\eps-1\le0$, and $s\mapsto\nabla^2f(z_s)$ is integrable even when the
segment meets the origin,
\begin{align*}
 f(a)-f(b)-\nabla f(b)\cdot(a-b)
 &=\int_0^1(1-s)\nabla^2f(z_s)(a-b,a-b)\dd s\\
 &\ge\eps(|a|+|b|)^{\eps-1}|a-b|^2
 \int_0^1(1-s)\dd s,
\end{align*}
which proves the claim.
\end{proof}

We record two consequences of the preceding estimate.
First, the mean value theorem gives
\[
 \bigl||a|^{1+\eps}-|b|^{1+\eps}\bigr|
 \le (1+\eps)(|a|+|b|)^\eps|a-b|.
\]
Integrating and applying H\"older's inequality with
exponents $(1+\eps)/\eps$ and $1+\eps$, we obtain
\begin{align*}
 &\bigl\||a|^{1+\eps}-|b|^{1+\eps}\bigr\|_{L^1}\\
 &\quad\le (1+\eps)
 \left(\int(|a|+|b|)^{1+\eps}\dd V\right)^{\frac{\eps}{1+\eps}}
 \left(\int|a-b|^{1+\eps}\dd V\right)^{\frac1{1+\eps}}\\
 &\quad\le (1+\eps)
 (\|a\|+\|b\|)^\eps\|a-b\|.
\end{align*}
Thus,
\begin{equation}\label{eq:power-L1}
 \bigl\||a|^{1+\eps}-|b|^{1+\eps}\bigr\|_{L^1}
 \le (1+\eps)(\|a\|+\|b\|)^\eps\|a-b\|.
\end{equation}

Next, suppose that
\[
 \int |b|^{\eps-1}b\cdot(a-b)\dd V=0.
\]
Integrating \eqref{eq:convexity} yields
\begin{equation}\label{eq:weighted-convexity}
 \int (|a|+|b|)^{\eps-1}|a-b|^2\dd V
 \le \frac{2}{\eps(1+\eps)}
 \bigl(\|a\|^{1+\eps}-\|b\|^{1+\eps}\bigr).
\end{equation}
On the other hand, for $0<\eps<1$,
H\"older's inequality with exponents
$2/(1+\eps)$ and $2/(1-\eps)$ gives the following.
\begin{align*}
 \|a-b\|^{1+\eps}
 &=\int
 \left[(|a|+|b|)^{\eps-1}|a-b|^2\right]^{\frac{1+\eps}{2}}
 (|a|+|b|)^{\frac{1-\eps^2}{2}}\dd V\\
 &\le
 \left(\int (|a|+|b|)^{\eps-1}|a-b|^2\dd V
 \right)^{\frac{1+\eps}{2}}
 \left(\int(|a|+|b|)^{1+\eps}\dd V
 \right)^{\frac{1-\eps}{2}}.
\end{align*}
Raising both sides to the power $2/(1+\eps)$
and using Minkowski's inequality, we find
\[
 \|a-b\|^2
 \le
 \left(\int (|a|+|b|)^{\eps-1}
 |a-b|^2\dd V\right)
 (\|a\|+\|b\|)^{1-\eps}.
\]
The same inequality holds directly when $\eps=1$.
Combining this with \eqref{eq:weighted-convexity},
we conclude that
\begin{equation}\label{eq:Lp-defect}
 \|a-b\|^2
 \le \frac{2}{\eps(1+\eps)}
 \bigl(\|a\|^{1+\eps}-\|b\|^{1+\eps}\bigr)
 (\|a\|+\|b\|)^{1-\eps}.
\end{equation}

% ===== sec3_local.tex =====
\section{Selection and local analysis}\label{sec:local}

By \eqref{eq:primal-full} and \eqref{eq:dual-full}, the finite coefficients
of the two bounds built from a lapse $v$ with $Jv>0$ differ by
\begin{equation}\label{eq:gap-Jensen}
 \Bigl(\frac12\log\frac{\E(vJv)}{\pi/2}-\E\log v\Bigr)
 -\Bigl(\frac12\E\log\frac{Jv}v-\frac12\log\frac\pi2\Bigr)
 =\frac12\bigl(\log\E(vJv)-\E\log(vJv)\bigr)\ \ge\ 0.
\end{equation}
The difference is half the Jensen gap for the logarithm
applied to $vJv$, and vanishes exactly when $vJv$ is
constant. Since both coefficients are invariant under
$v\mapsto\lambda v$, we may normalize this constant to one.
Thus the two bounds agree at order $\eps$ precisely when
$v$ is a positive constant multiple of a solution of
$Jg=g^{-1}$.

\begin{lemma}\label{lem:selection}
The equation $Jg=g^{-1}$ admits a unique smooth positive
solution $g$ on $\Sigma$.
\end{lemma}
For $\Sigma\cong S^2$ this is proved in \cite{Smith2011}; we include a
proof for an arbitrary closed $\Sigma$.

\begin{proof}
Let $\varphi>0$ be the first eigenfunction of $J$.
Since $\lambda_1(J)>0$, for sufficiently small $\delta>0$,
the functions $\delta\varphi$ and $\delta^{-1}\varphi$
satisfy
\[
 J(\delta\varphi)\le(\delta\varphi)^{-1},
 \qquad
 J(\delta^{-1}\varphi)\ge(\delta^{-1}\varphi)^{-1}.
\]
The standard sub- and supersolution method, followed by
elliptic regularity, gives a smooth positive solution.

For uniqueness, suppose $g_1$ and $g_2$ are two positive
solutions. Subtracting their equations gives
\[
 J(g_1-g_2)=-\frac{g_1-g_2}{g_1g_2}.
\]
Multiplying by $g_1-g_2$ and integrating, we obtain
\[
 0=\int_\Sigma (g_1-g_2)J(g_1-g_2)\dd A
   +\int_\Sigma\frac{(g_1-g_2)^2}{g_1g_2}\dd A.
\]
Both terms are nonnegative by $\lambda_1(J)>0$,
so $g_1=g_2$.
\end{proof}

From now on $g$ is this solution and the coordinates
\eqref{eq:collar-metric} are built from $v=g$. Then $\ell(0,\cdot)=g$,
$gJg\equiv1$ and, by \eqref{eq:collar-expansions},
\begin{equation}\label{eq:g-collar}
 \ell=g\bigl(1+b_1t+O(t^2)\bigr),\qquad w=\tfrac12t^2+a_3t^3+O(t^4);
\end{equation}
after shrinking $d_1$, $ct^2\le w\le Ct^2$ and $w_t\ge ct$.

\subsection{Matching of the two bounds}\label{sec:matching}
Put
\begin{equation}\label{eq:nu-eps}
 d\nu_\eps=\frac{g^{-\eps}\dd A}{\int_\Sigma g^{-\eps}\dd A},
 \qquad
 \log\frac{\int_\Sigma g^{-\eps}\dd A}{A}
 =-\eps\E\log g+O(\eps^2).
\end{equation}
Then $\mathbb E_{\nu_\eps}\phi=\E\phi+O(\eps)$
for every fixed smooth function $\phi$ on $\Sigma$.
We also write
\[
 T=\int_U|X|^{(1+\eps)/\eps}\dd V,
 \qquad
 S=\int_0^d
 \Bigl(\int_\Sigma\ell^{-\eps}e^w\dd A\Bigr)^{-1/\eps}\dd t,
\]
so that $\Clower=T^{-\eps}$ and $\Cupper=S^{-\eps}$.

\emph{The dual bound.} For $v=g$, \eqref{eq:flux-field} gives
$|X|=g^{-\eps}e^{-w}\big/\int_\Sigma g^{-\eps}\dd A$, and $dV=\ell e^w\dd t\dd A$, so
\begin{equation}\label{eq:dual-integral-g}
 T=\Bigl(\int_\Sigma g^{-\eps}\dd A\Bigr)^{-\frac{1+\eps}\eps}
 \int_\Sigma g^{-1-\eps}\int_0^d\ell\,e^{-w/\eps}\dd t\dd A .
\end{equation}
In the inner integral put $t=\sqrt\eps\,y$. By \eqref{eq:g-collar},
\[
 \frac w\eps=\frac{y^2}2+\sqrt\eps\,a_3y^3+O(\eps y^4),
 \qquad
 \frac\ell g=1+\sqrt\eps\,b_1y+O(\eps y^2),
\]
so with
\[
 \int_0^\infty e^{-y^2/2}\dd y=\sqrt{\frac\pi2},
 \qquad
 \int_0^\infty ye^{-y^2/2}\dd y=1,
 \qquad
 \int_0^\infty y^3e^{-y^2/2}\dd y=2
\]
and Lemma~\ref{lem:laplace} for the remainders,
\begin{align}
 \int_0^d\ell\,e^{-w/\eps}\dd t
 &=g\sqrt\eps\Bigl[
 \int_0^\infty e^{-y^2/2}
 \bigl(1+\sqrt\eps\,(b_1y-a_3y^3)\bigr)\dd y
 +O(\eps)\Bigr]\notag\\
 &=g\sqrt\eps\,\Bigl[
 \sqrt{\tfrac\pi2}
 +\sqrt\eps\,(b_1-2a_3)
 +O(\eps)\Bigr]\label{eq:fiber-dual}
\end{align}
uniformly in $x$. In \eqref{eq:dual-integral-g} the factor
$g^{-1-\eps}\cdot g=g^{-\eps}$ turns the outer integral into a
$\nu_\eps$-average, and $1-\frac{1+\eps}\eps=-\frac1\eps$:
\begin{align}
 T&=\Bigl(\int_\Sigma g^{-\eps}\dd A\Bigr)^{-\frac{1+\eps}\eps}\sqrt\eps
 \int_\Sigma g^{-\eps}\Bigl[\sqrt{\tfrac\pi2}+\sqrt\eps\,(b_1-2a_3)+O(\eps)\Bigr]\dd A\notag\\
 &=\Bigl(\int_\Sigma g^{-\eps}\dd A\Bigr)^{-\frac1\eps}\sqrt\eps\,
 \Bigl[\sqrt{\tfrac\pi2}+\sqrt\eps\,\mathbb E_{\nu_\eps}(b_1-2a_3)+O(\eps)\Bigr].\label{eq:T-expansion}
\end{align}

\emph{The primal bound.} Write
\[
 \int_\Sigma\ell^{-\eps}e^w\dd A=\Bigl(\int_\Sigma g^{-\eps}\dd A\Bigr)\,e^{\Phi(t)}
\]
with $\Phi=\log\mathbb E_{\nu_\eps}e^{\psi}$, $\psi=w-\eps\log(\ell/g)$. By
\eqref{eq:g-collar},
\[
 \psi=\tfrac12t^2+a_3t^3-\eps b_1t+O(t^4+\eps t^2),
\]
and its $x$-dependent part is $O(t^3+\eps t)$, so
$\log\mathbb E_{\nu_\eps}e^\psi=\mathbb E_{\nu_\eps}\psi+O(t^6+\eps^2t^2)$ and
\begin{equation}\label{eq:primal-phase}
 \Phi(t)=\tfrac12t^2+\mathbb E_{\nu_\eps}a_3\,t^3-\eps\,\mathbb E_{\nu_\eps}b_1\,t+O(t^4+\eps t^2).
\end{equation}
The same computation in $t=\sqrt\eps\,y$ gives
\begin{equation}\label{eq:S-expansion}
 S=\Bigl(\int_\Sigma g^{-\eps}\dd A\Bigr)^{-\frac1\eps}\int_0^de^{-\Phi/\eps}\dd t
 =\Bigl(\int_\Sigma g^{-\eps}\dd A\Bigr)^{-\frac1\eps}\sqrt\eps\,
 \Bigl[\sqrt{\tfrac\pi2}+\sqrt\eps\,\mathbb E_{\nu_\eps}(b_1-2a_3)+O(\eps)\Bigr].
\end{equation}
The first corrections in \eqref{eq:T-expansion} and \eqref{eq:S-expansion}
coincide because both are $\nu_\eps$-averages of $b_1-2a_3$. Hence
$T/S=1+O(\eps)$ and, since $\Clower\le C_U\le\Cupper$,
\begin{equation}\label{eq:matching-ratio}
 1\le\frac{C_U}{\Clower}\le\frac{\Cupper}{\Clower}=\Bigl(\frac TS\Bigr)^{\eps}=1+O(\eps^2).
\end{equation}
Taking $-\eps\log$ in \eqref{eq:T-expansion} and using \eqref{eq:nu-eps},
\begin{align}
 \log\frac{\Clower}A
 &=\log\frac{\int_\Sigma g^{-\eps}\dd A}A-\frac\eps2\log\frac{\pi\eps}2
 -\eps\log\Bigl[1+\sqrt{\tfrac{2\eps}\pi}\,\mathbb E_{\nu_\eps}(b_1-2a_3)+O(\eps)\Bigr]\notag\\
 &=\frac\eps2|\log\eps|+\eps\Bigl(-\frac12\log\frac\pi2-\E\log g\Bigr)+\kappa_+\eps^{3/2}+O(\eps^2),
 \qquad \kappa_+=\sqrt{\tfrac2\pi}\,\E(2a_3-b_1),\label{eq:one-sided-refined}
\end{align}
and by \eqref{eq:matching-ratio} the same expansion holds for $C_U$ and
$\Cupper$. The coefficients come from the expansions at $t=0$ only.

\emph{The coefficient $\kappa_+$.}
The collar leaves are the geodesic normal graphs
\[
 \Sigma_t=\{\exp_x(tg(x)N(x)):x\in\Sigma\}.
\]
In normal distance coordinates $(r,x)$, these leaves are
given by $t=r/g(x)$. A direct computation gives
\[
 \ell^{-2}=|\nabla t|^2=g^{-2}+O(t^2),
\]
so that $\ell=g+O(t^2)$ and hence $b_1=0$.
Since $w=\frac12t^2+a_3t^3+O(t^4)$, we have
\[
 |\Sigma_t|=\int_\Sigma e^w\dd A
 =A+\frac A2t^2+A\,\E a_3\,t^3+O(t^4).
\]
Thus $\frac{d^3}{dt^3}|_{t=0}|\Sigma_t|=6A\,\E a_3$,
and \eqref{eq:one-sided-refined} yields
\begin{equation}\label{eq:kappa-area}
 \kappa_+=\frac1{3A}\sqrt{\frac2\pi}\,
 \left.\frac{d^3}{dt^3}|\Sigma_t|\right|_{t=0}.
\end{equation}
For an interior hypersurface, reversing the normal direction
($g$ is unchanged) changes the sign of
$\kappa_+$. In particular, reflection symmetry across
$\Sigma$ implies $\kappa_+=0$.

\emph{The comparison field.} For $v=g$, \eqref{eq:b-scalar-general} reads
$|\bb|=(\int_\Sigma g^{-\eps}\dd A)^{-1/\eps}g^{-1}e^{-w/\eps}\big/T$, and by
\eqref{eq:T-expansion}
\begin{align}
 \sqrt\eps\,|\bb|(t,x)
 &=\sqrt{\frac2\pi}\;\frac{e^{-w(t,x)/\eps}}{g(x)}\,
 \Bigl[1+\sqrt{\tfrac{2\eps}\pi}\,\mathbb E_{\nu_\eps}(b_1-2a_3)+O(\eps)\Bigr]^{-1}\notag\\
 &=\sqrt{\frac2\pi}\;\frac{e^{-w(t,x)/\eps}}{g(x)}\,\bigl(1+\kappa_+\sqrt\eps+c_\eps\bigr),\label{eq:b-scalar}
\end{align}
where $c_\eps=O(\eps)$ is a constant, independent of $(t,x)$. With
$\ell/g=1+b_1\sqrt\eps\,y+O(\eps y^2)$ and
$e^{-w/\eps}=e^{-y^2/2}\bigl(1-a_3\sqrt\eps\,y^3+O(\eps(y^4+y^6))\bigr)$ this gives, on fixed $y$-intervals,
\begin{equation}\label{eq:kernel-expansion}
 \sqrt\eps\,\ell\,|\bb|(\sqrt\eps\,y,x)
 =\gamma_+(y)\Bigl[1+\sqrt\eps\bigl(b_1(x)\,y-a_3(x)\,y^3+\kappa_+\bigr)+O\bigl(\eps(1+y^6)\bigr)\Bigr];
\end{equation}
the first correction has fiber mass
$\int_0^\infty(b_1y-a_3y^3+\kappa_+)\gamma_+\dd y=\sqrt{2/\pi}\,(b_1-2a_3)+\kappa_+$,
whose $\nu$-average is zero. The weighted estimate of
Lemma~\ref{lem:laplace} makes \eqref{eq:kernel-expansion} global in $y$: for
every fixed integer $k\ge0$, with the kernel extended by zero for
$\sqrt\eps\,y\ge d$,
\begin{equation}\label{eq:kernel-L1}
 \sup_{x\in\Sigma}\int_0^\infty(1+y^k)\,
 \bigl|\sqrt\eps\,\ell(\sqrt\eps\,y,x)\,|\bb|(\sqrt\eps\,y,x)-\gamma_+(y)\bigr|\dd y
 \le C_k\sqrt\eps .
\end{equation}

\emph{The comparison measure.}
By \eqref{eq:b-identities},
$\Clower^{-1}|\bb|^{1+\eps}\dd V$ is a probability
measure on $U$. Since $|\bb|^\eps=\Clower|X|$,
the expressions for $X$ and $dV$ give
\[
 \Clower^{-1}|\bb|^{1+\eps}\dd V
 =|X|\,|\bb|\,\ell e^w\dd t\dd A
 =\ell|\bb|\dd t\,d\nu_\eps.
\]
After $t=\sqrt\eps\,y$, its density with respect to
$d\nu\dd y$ is
\begin{equation}\label{eq:comparison-density}
 \vartheta_\eps(x,y)
 =\frac{g(x)^{-\eps}}{\E(g^{-\eps})}\,
   \sqrt\eps\,\ell(\sqrt\eps\,y,x)
   |\bb|(\sqrt\eps\,y,x)
 \le Ce^{-cy^2},
\end{equation}
for $0\le y<d/\sqrt\eps$, and is extended by zero
elsewhere. The upper bound follows from
$w\ge ct^2$ and \eqref{eq:b-scalar}.
Since $g^{-\eps}/\E(g^{-\eps})=1+O(\eps)$
uniformly on $\Sigma$, \eqref{eq:kernel-L1} gives,
for every fixed integer $k\ge0$,
\begin{equation}\label{eq:density-distance}
 \int_\Sigma\int_0^\infty
 (1+y^k)\,
 |\vartheta_\eps(x,y)-\gamma_+(y)|
 \dd y\dd\nu(x)
 \le C_k\sqrt\eps .
\end{equation}

\subsection{Selection of the actual minimizer}\label{sec:minimizer}
The comparison field makes the linear term in the convexity estimate
vanish for every competitor: if $u$ is admissible in
\eqref{eq:capacity}, then by \eqref{eq:b-definition},
\eqref{eq:flux-identity} and \eqref{eq:b-identities},
\[
 \int_U|\bb|^{\eps-1}\bb\cdot(-\nabla u-\bb)\dd V
 =\Clower\Bigl(-\int_UX\cdot\nabla u-\int_UX\cdot\bb\Bigr)=\Clower\,(1-1)=0 .
\]
Applying \eqref{eq:Lp-defect} with $a=-\nabla u$ and $b=\bb$ gives
\begin{equation}\label{eq:near-min}
 \|\nabla u+\bb\|_{L^{1+\eps}(U)}^2
 \le\frac{2\bigl(\int_U|\nabla u|^{1+\eps}\dd V-\Clower\bigr)}{\eps(1+\eps)}\,
 \bigl(\|\nabla u\|_{L^{1+\eps}}+\|\bb\|_{L^{1+\eps}}\bigr)^{1-\eps}:
\end{equation}
the gradient of a competitor is close to $-\bb$ as soon as its energy is
close to the lower bound. For the minimizer the energy excess is
$O(\eps^2)$ by \eqref{eq:matching-ratio}, which gives the following.

\begin{proposition}[Gaussian selection of the minimizer]
\label{prop:selection-local}
As $\eps\downarrow0$, the following estimates hold.
\begin{enumerate}
\item[\textup{(i)}] \textup{(Gradient.)}
\begin{equation*}
 \|\nabla v_\eps+\bb\|_{L^{1+\eps}(U)}
 \le C\sqrt\eps.
\end{equation*}

\item[\textup{(ii)}] \textup{(Energy.)}
\begin{equation*}
 \bigl\|(R_\eps)_\#\mu_\eps^U
 -\nu\otimes\gamma_+(y)\dd y\bigr\|_{\mathrm{TV}}
 \le C\sqrt\eps.
\end{equation*}

\item[\textup{(iii)}] \textup{(Profile.)}
With $\widetilde v_\eps(y,x)=v_\eps(\sqrt\eps\,y,x)$
for $\sqrt\eps\,y<d$ and zero otherwise,
\begin{equation}\label{eq:main-profile}
 \sup_{x\in\Sigma}\sup_{y\ge0}
 |\widetilde v_\eps(y,x)-P_+(y)|
 \le C\sqrt\eps.
\end{equation}

\item[\textup{(iv)}] \textup{(Boundary gradient.)}
\begin{equation}\label{eq:main-normal-derivative}
 \sup_{\Sigma}\Bigl|\sqrt\eps\,g\,|\nabla v_\eps|-\sqrt{\tfrac2\pi}\Bigr|
 \le C\sqrt\eps.
\end{equation}
\end{enumerate}
Here $\|\alpha-\beta\|_{\mathrm{TV}}
=\sup_B|\alpha(B)-\beta(B)|$.
\end{proposition}

\begin{proof}
(i) By \eqref{eq:near-min} with $u=v_\eps$,
\[
 \|\nabla v_\eps+\bb\|_{L^{1+\eps}(U)}^2
 \le \frac{2(C_U-\Clower)}{\eps(1+\eps)}
 \bigl(\|\nabla v_\eps\|_{L^{1+\eps}}
       +\|\bb\|_{L^{1+\eps}}\bigr)^{1-\eps}.
\]
Since $C_U-\Clower=O(\eps^2)$ by
\eqref{eq:matching-ratio}, and both norms are
uniformly bounded, the right-hand side is $O(\eps)$.
This proves (i).

(ii) By \eqref{eq:power-L1} and (i),
\[
 \bigl\||\nabla v_\eps|^{1+\eps}
       -|\bb|^{1+\eps}\bigr\|_{L^1(U)}
 \le C\sqrt\eps.
\]
Since $\mu_\eps^U$ and
$\Clower^{-1}|\bb|^{1+\eps}\dd V$
are probability measures, we obtain
\begin{align*}
 \bigl\|\mu_\eps^U
 -\Clower^{-1}|\bb|^{1+\eps}\dd V
 \bigr\|_{\mathrm{TV}}
 \le
 \frac{1}{2C_U}
 \bigl\||\nabla v_\eps|^{1+\eps}
       -|\bb|^{1+\eps}\bigr\|_{L^1(U)}
 +\frac{|C_U-\Clower|}{2C_U}\le C\sqrt\eps.
\end{align*}
Push-forward does not increase total variation,
and \eqref{eq:density-distance} with $k=0$
shows that the rescaled comparison measure is
$O(\sqrt\eps)$-close to $\nu\otimes\gamma_+(y)\dd y$.
The triangle inequality proves (ii).

(iii) Let $f_\pm$ be as in \eqref{eq:barriers}. By Lemma~\ref{lem:barriers},
$f_+(\sqrt\eps\,y)\le\widetilde v_\eps(y,x)\le f_-(\sqrt\eps\,y)$ for all
$y\ge0$ and $x\in\Sigma$, and \eqref{eq:barrier-asymptotics} gives
\eqref{eq:main-profile}.

(iv) On $\Sigma$, $|\nabla v_\eps|=g^{-1}|\partial_tv_\eps|$, and
$f_+\le v_\eps\le f_-$ with equality at $t=0$ gives
$|f_-'(0)|\le|\partial_tv_\eps(0,x)|\le|f_+'(0)|$; now use
\eqref{eq:barrier-asymptotics}.
\end{proof}

% ===== sec3b_moments.tex =====
\subsection{Logarithmic moments}\label{sec:moments}
The observable $\log|\nabla v_\eps|$ is unbounded where the gradient is small
or large. We first bound the gradient, then transfer logarithmic moments
from the comparison field to the minimizer. Throughout,
$s^{1+\eps}\log^ks$ is taken to be zero at $s=0$.

\begin{lemma}[Gradient bound]\label{lem:gradient}
$\|\nabla v_\eps\|_{L^\infty(U)}\le C/\sqrt\eps$.
\end{lemma}

\begin{proof}
On $\partial U$, $|\nabla v_\eps|=\ell^{-1}|\partial_tv_\eps|$, and
$f_+\le v_\eps\le f_-$ from \eqref{eq:barrier-sandwich}, with equality on
$\Sigma$ and $\Sigma_d$, gives
\[
 |\partial_tv_\eps(0,x)|\le|f_+'(0)|,
 \qquad
 |\partial_tv_\eps(d,x)|\le|f_-'(d)|\le|f_-'(0)|.
\]
By \eqref{eq:barrier-asymptotics}, $\max_{\partial U}|\nabla v_\eps|\le C\eps^{-1/2}$.

Put $s=|\nabla v_\eps|$. On $\{s>0\}$, consider the elliptic
operator
\[
 L_\eps\phi
 =\diver\left[
 s^{\eps-1}\left(
 \nabla\phi+(\eps-1)
 \frac{\langle\nabla v_\eps,\nabla\phi\rangle}{s^2}
 \nabla v_\eps
 \right)\right],
 \qquad L_\eps v_\eps=0.
\]
Fix $\Lambda\ge0$ with $\Ric\ge-\Lambda\mathfrak g$ on $U$.
On $\{s>0\}$, the Bochner formula
(cf.~\cite{KotschwarNi2009}) gives
\[
 \frac12L_\eps(s^2)
 =s^{\eps-1}\Bigl(
 |\nabla^2v_\eps|^2+(\eps-1)|\nabla s|^2
 +\Ric(\nabla v_\eps,\nabla v_\eps)
 \Bigr).
\]
By Kato's inequality $|\nabla^2v_\eps|^2\ge|\nabla s|^2$
and the Ricci lower bound,
\[
 \frac12L_\eps(s^2)
 \ge s^{\eps-1}
 \bigl(\eps|\nabla s|^2-\Lambda s^2\bigr).
\]

Let $Q=s^2/(2-v_\eps)$, and suppose that its maximum over
$\overline U$ is positive and attained at an interior point.
At that point $s>0$, $\nabla Q=0$ and $\nabla^2Q\le0$,
so $L_\eps Q\le0$. Differentiating $s^2=(2-v_\eps)Q$ gives
\[
 \nabla(s^2)=-Q\nabla v_\eps,
 \qquad
 |\nabla s|^2
 =\frac{Q^2}{4s^2}|\nabla v_\eps|^2
 =\frac{Q^2}{4}.
\]
Using $L_\eps v_\eps=0$ and the product rule, we obtain
\[
\begin{aligned}
 0\ge L_\eps Q
 =
 \frac{L_\eps(s^2)}{2-v_\eps}
 -\frac{2\eps Qs^{\eps+1}}{(2-v_\eps)^2}
 +\frac{2\eps s^{\eps+3}}{(2-v_\eps)^3}
\ge
 \frac{2s^{\eps-1}}{2-v_\eps}
 \left(\frac{\eps Q^2}{4}-\Lambda s^2\right).
\end{aligned}
\]
The two quotient terms cancel because
$s^2=(2-v_\eps)Q$. Hence
\[
 \eps Q^2\le4\Lambda s^2
 =4\Lambda(2-v_\eps)Q\le8\Lambda Q.
\]
Since $Q>0$ at the maximum, this implies
$\max_{\overline U}Q\le8\Lambda/\eps$, and therefore
\[
 s^2=(2-v_\eps)Q
 \le2\max_{\overline U}Q
 \le\frac{16\Lambda}{\eps}
 \qquad\text{on }U.
\]
If the maximum of $Q$ is attained on $\partial U$, then
\[
 \max_{\overline U}Q
 \le\max_{\partial U}s^2,
 \qquad
 s^2\le2\max_{\partial U}s^2
 \qquad\text{on }U.
\]
Combining the two cases with the boundary estimate yields
\[
 \|\nabla v_\eps\|_{L^\infty(U)}
 \le
 \max\left\{
 \sqrt2\max_{\partial U}|\nabla v_\eps|,
 4\sqrt{\Lambda/\eps}
 \right\}
 \le\frac C{\sqrt\eps}.
 \qedhere
\]
\end{proof}

\begin{lemma}[Transfer of logarithmic moments]
\label{lem:transfer}
Let $0<\eps\le e^{-1}$ and $k\ge1$ be an integer.
Suppose that $s,b\ge0$ are measurable on $U$ and
\[
 \|s\|_{L^\infty(U)}+\|b\|_{L^\infty(U)}
 \le C_0\eps^{-1}
\]
for a fixed $C_0>0$. Then
\begin{equation}\label{eq:transfer}
\begin{aligned}
 &\left|\int_U
 \left(s^{1+\eps}\log^k(\sqrt\eps\,s)
 -b^{1+\eps}\log^k(\sqrt\eps\,b)\right)\dd V\right|\\
 &\quad\le C_k|\log\eps|^k
 \Bigl[
 \bigl(\|s\|+\|b\|
 +\eps|U|^{1/(1+\eps)}\bigr)^\eps
 \|s-b\|+\eps^{1+\eps}|U|
 \Bigr],
\end{aligned}
\end{equation}
where the norms without subscripts are in
$L^{1+\eps}(U)$ and $C_k$ depends only on
$k$ and $C_0$.
\end{lemma}

\begin{proof}
We aim to control the difference of logarithmic moments
by the $L^{1+\eps}$ distance between $s$ and $b$.
Let $\Phi(z)=z^{1+\eps}\log^k(\sqrt\eps\,z)$,
with $\Phi(0)=0$.

\emph{Case 1: $0\le z\le\eps$.}
Writing $z=\eps\vartheta$, $0<\vartheta\le1$, gives
\[
 |\Phi(z)|
 \le \eps^{1+\eps}\vartheta^{1+\eps}
 \left(\frac32|\log\eps|+|\log\vartheta|\right)^k
 \le C_k\eps^{1+\eps}|\log\eps|^k,
\]
since $\vartheta^{1+\eps}|\log\vartheta|^k$
is uniformly bounded. Thus the contribution from
values below $\eps$ is uniformly small.

\emph{Case 2: $z\ge\eps$.}
On $\eps\le z\le\max\{\eps,C_0\eps^{-1}\}$,
we have
\[
 |\log(\sqrt\eps\,z)|\le C|\log\eps|.
\]
Differentiating $\Phi$ therefore gives
\[
 |\Phi'(z)|
 \le C_k|\log\eps|^k z^\eps.
\]
Hence the mean value theorem applies with a
uniform bound on the derivative.

\emph{Combining the estimates.}
Set $\bar s=\max\{s,\eps\}$ and
$\bar b=\max\{b,\eps\}$. By Case 1,
\[
 |\Phi(s)-\Phi(\bar s)|
 +|\Phi(b)-\Phi(\bar b)|
 \le C_k\eps^{1+\eps}|\log\eps|^k.
\]
By Case 2 and the mean value theorem,
\[
 |\Phi(\bar s)-\Phi(\bar b)|
 \le C_k|\log\eps|^k
 (s+b+\eps)^\eps|s-b|,
\]
where we used
$|\bar s-\bar b|\le|s-b|$.
Combining these inequalities yields
\[
 |\Phi(s)-\Phi(b)|
 \le C_k|\log\eps|^k
 \left((s+b+\eps)^\eps|s-b|
 +\eps^{1+\eps}\right).
\]

Finally, H\"older's inequality with exponents
$(1+\eps)/\eps$ and $1+\eps$ gives
\begin{align*}
 \int_U(s+b+\eps)^\eps|s-b|\dd V\le
 \|s+b+\eps\|^\eps\|s-b\|\le
 \bigl(\|s\|+\|b\|
 +\eps|U|^{1/(1+\eps)}\bigr)^\eps
 \|s-b\|,
\end{align*}
where the last step follows from Minkowski's
inequality. Integrating the preceding pointwise
estimate proves \eqref{eq:transfer}.
\end{proof}

Let
\[
 Y(x,y)=-\frac12\log\frac\pi2-\log g(x)-\frac{y^2}2 ,
\]
regarded as a random variable under $\nu\otimes\gamma_+(y)\dd y$. Since $x$
and $y$ are independent, $\int_0^\infty y^2\gamma_+\dd y=1$ and
$\int_0^\infty y^4\gamma_+\dd y=3$,
\begin{equation}\label{eq:limit-law}
 \mathbb E[Y]=-\frac12\log\frac\pi2-\E\log g-\frac12,\qquad
 \Var[Y]=\Var_\nu(\log g)+\Var\Bigl(\frac{y^2}2\Bigr)=\Var_\nu(\log g)+\frac12 .
\end{equation}

\begin{proposition}[Logarithmic mean and variance]\label{prop:moments}
For $k=1,2$,
\begin{equation}\label{eq:main-moments}
 \int_U\log^k\bigl(\sqrt\eps\,|\nabla v_\eps|\bigr)\dd\mu_\eps^U=\mathbb E[Y^k]+O\bigl(\sqrt\eps\,|\log\eps|^k\bigr),
\end{equation}
and consequently
\begin{align}
 \int_U\log|\nabla v_\eps|\dd\mu_\eps^U
 &=\frac12|\log\eps|-\frac12\log\frac\pi2-\E\log g-\frac12+O(\sqrt\eps|\log\eps|),
 \label{eq:main-first-moment}\\
 \Var_{\mu_\eps^U}\bigl(\log|\nabla v_\eps|\bigr)
 &=\frac12+\Var_\nu(\log g)+O(\sqrt\eps|\log\eps|^2).\label{eq:main-variance}
\end{align}
\end{proposition}
\begin{proof}
We first compute the logarithmic moments of the comparison field.
We then transfer these moments to the minimizer using
Lemma~\ref{lem:transfer}.

To identify the comparison moments, we use the explicit profile
\eqref{eq:b-scalar}. Together with
$w(t,x)=t^2/2+O(t^3)$, it gives
\[
\begin{aligned}
 \log\bigl(\sqrt\eps\,|\bb|(\sqrt\eps\,y,x)\bigr)
 &=-\frac12\log\frac\pi2-\log g(x)
   -\frac{w(\sqrt\eps\,y,x)}{\eps}+O(\sqrt\eps)\\
 &=Y(x,y)+O\bigl(\sqrt\eps(1+y^3)\bigr)
\end{aligned}
\]
for $0<y<d/\sqrt\eps$.
Moreover, $w(t,x)\le Ct^2$ implies
\[
 |Y(x,y)|
 +\bigl|\log(\sqrt\eps\,|\bb|(\sqrt\eps\,y,x))\bigr|
 \le C(1+y^2).
\]
Consequently, for $k=1,2$,
\[
 \bigl|\log^k(\sqrt\eps\,|\bb|)-Y^k\bigr|
 \le C_k\sqrt\eps(1+y^{2k+1}),
\]
where, here and in the rescaled integrals below, $\bb$ is
evaluated at $(\sqrt\eps\,y,x)$.

The comparison measure
$\Clower^{-1}|\bb|^{1+\eps}\dd V$ has rescaled density
$\vartheta_\eps$ by \eqref{eq:comparison-density}.
To compare its logarithmic moment with $\mathbb E[Y^k]$,
we separate the error in the logarithm from the error in the density:
\begin{align*}
 &\left|
 \frac1{\Clower}\int_U
 |\bb|^{1+\eps}\log^k(\sqrt\eps\,|\bb|)\dd V
 -\mathbb E[Y^k]\right|\\
 &\quad\le
 \int_\Sigma\int_0^{d/\sqrt\eps}
 \bigl|\log^k(\sqrt\eps\,|\bb|)-Y^k\bigr|
 \vartheta_\eps\dd y\dd\nu
 +
 \int_\Sigma\int_0^\infty
 |Y|^k|\vartheta_\eps-\gamma_+|\dd y\dd\nu.
\end{align*}
The first integral is $O(\sqrt\eps)$ because
$\vartheta_\eps\le Ce^{-cy^2}$ and all Gaussian moments are finite.
The second is $O(\sqrt\eps)$ by
\eqref{eq:density-distance} with weight $1+y^{2k}$,
since $|Y|^k\le C_k(1+y^{2k})$.
Thus
\[
 \frac1{\Clower}\int_U
 |\bb|^{1+\eps}\log^k(\sqrt\eps\,|\bb|)\dd V
 =\mathbb E[Y^k]+O(\sqrt\eps).
\]

We now transfer this estimate to the minimizer.
The bounds required by Lemma~\ref{lem:transfer} follow from
Lemma~\ref{lem:gradient} and \eqref{eq:b-scalar}:
\[
 |\nabla v_\eps|\le C/\sqrt\eps,
 \qquad
 |\bb|\le C\eps^{-1/2}.
\]
Furthermore, Proposition~\ref{prop:selection-local}(i) gives
\[
 \bigl\|\,|\nabla v_\eps|-|\bb|\,\bigr\|_{L^{1+\eps}(U)}
 \le\|\nabla v_\eps+\bb\|_{L^{1+\eps}(U)}
 \le C\sqrt\eps.
\]
Both $L^{1+\eps}$ norms are uniformly bounded, since
their $(1+\eps)$-powers are $C_U$ and $\Clower$,
respectively, and both capacities tend to $A$.
Applying Lemma~\ref{lem:transfer} with
$s=|\nabla v_\eps|$ and $b=|\bb|$, we therefore obtain
\begin{align*}
 &\left|
 \int_U|\nabla v_\eps|^{1+\eps}
 \log^k(\sqrt\eps\,|\nabla v_\eps|)\dd V
 -
 \int_U|\bb|^{1+\eps}
 \log^k(\sqrt\eps\,|\bb|)\dd V
 \right|\\
 &\qquad\le C_k\sqrt\eps\,|\log\eps|^k.
\end{align*}

Dividing by $C_U$ and using the definition
$d\mu_\eps^U=C_U^{-1}|\nabla v_\eps|^{1+\eps}\dd V$,
we substitute the comparison moments computed above to obtain
\begin{align*}
 \int_U
 \log^k(\sqrt\eps\,|\nabla v_\eps|)\dd\mu_\eps^U
 &=
 \frac{\Clower}{C_U}
 \bigl(\mathbb E[Y^k]+O(\sqrt\eps)\bigr)
 +O\bigl(\sqrt\eps\,|\log\eps|^k\bigr)\\
 &=
 \mathbb E[Y^k]
 +O\bigl(\sqrt\eps\,|\log\eps|^k\bigr),
\end{align*}
where the last equality uses
$\Clower/C_U=1+O(\eps^2)$ from
\eqref{eq:matching-ratio}.
This proves \eqref{eq:main-moments}.

The mean and variance formulas follow from
\eqref{eq:main-moments} and \eqref{eq:limit-law},
since adding $\frac12|\log\eps|$ to
$\log(\sqrt\eps\,|\nabla v_\eps|)$ shifts its mean
and leaves its variance unchanged.
\end{proof}

\begin{remark}[Mean-convex comparison]
For comparison, suppose that $\Sigma$ has strictly
positive mean curvature $H$. Choosing the initial
lapse $H^{-1}$ gives $w_t(0,\cdot)=1$, and hence
\[
 w(t,x)=t+O(t^2).
\]
The same primal--dual argument, now with the
rescaling $t=\eps y$ and the flux field
$X=\bigl(\int_\Sigma H^\eps\dd A\bigr)^{-1}H^\eps e^{-w}\,\mathbf n$, gives
\[
 \log\frac{C_{U_H}}A
 =\eps|\log\eps|+\eps\E\log H+O(\eps^2),
\]
where $U_H$ is the corresponding collar.
The rescaled normalized energy measures converge
in total variation to
\[
 \nu\otimes e^{-y}\dd y,\qquad y\ge0.
\]
Thus the boundary layer has normal width $\eps/H$
and an exponential limiting law, in contrast
to the half-Gaussian layer of width $g\sqrt\eps$
at a strictly stable minimal boundary.
\end{remark}

% ===== sec4_global.tex =====
\section{Globalization and proof of Theorem~\ref{thm:global}}\label{sec:global}

Throughout this section the hypotheses of Theorem~\ref{thm:global} hold,
the collar coordinates are built from $g$, $U=(0,d)\times\Sigma$ and
$\Sigma_d=\{d\}\times\Sigma$ are as in Section~\ref{sec:coordinates}, $v_\eps$ is the
collar minimizer and $\bar v_\eps$ its extension by zero. We write
\[
 u_\eps=u_{1+\eps},\qquad C_M=C_M(1+\eps),\qquad d\mu_\eps=C_M^{-1}|\nabla u_\eps|^{1+\eps}\dd V ;
\]
the potential $u_\eps$ exists, $0<u_\eps<1$, and
$C_M=\int_M|\nabla u_\eps|^{1+\eps}\dd V$
(\cite{FogagnoloMazzieri2022}, after smoothly filling in $\Sigma$), so $\mu_\eps$ is a
probability measure, and $\bar v_\eps$ is a competitor for $C_M$. 

\subsection{The logarithmic limit}
The functions $z_\eps=-\eps\log u_\eps$ converge to the arrival time of the
weak inverse mean curvature flow from $\Sigma$. Since $\Sigma$ is strictly
outward minimizing, this arrival time is positive away from $\Sigma$, and
hence $u_\eps$ is exponentially small there.

\begin{lemma}[The logarithmic limit]\label{lem:log-control}
Let $z_\eps=-\eps\log u_\eps$.
\begin{enumerate}
\item[\textup{(a)}] $\|\nabla z_\eps\|_{L^\infty(M)}\le C$, and for a fixed
large $R$ in the asymptotically flat chart
\begin{equation}\label{eq:AF-barrier}
 u_\eps\le(R/r)^{(n-2)/\eps}\qquad(r\ge R).
\end{equation}
\item[\textup{(b)}] As $\eps\downarrow0$, $z_\eps\to z$ locally uniformly,
where $z$ is the proper weak solution of the inverse mean curvature flow
with initial region bounded by $\Sigma$.
\item[\textup{(c)}] $z>0$ on $M\setminus\Sigma$. Consequently, for every
compact $S'\subset M\setminus\Sigma$,
\begin{equation}\label{eq:trace-log-limit}
 \lim_{\eps\downarrow0}\ -\eps\log\max_{S'}u_\eps=\min_{S'}z>0 .
\end{equation}
\end{enumerate}
\end{lemma}
\begin{proof}
After smoothly filling in $\Sigma$, the gradient estimate follows
from \cite[Theorem~2.24]{MariRigoliSetti2022}.
For a fixed sufficiently large $R$ and all small $\eps$,
$(R/r)^{(n-2)/\eps}$ is $(1+\eps)$-superharmonic on $\{r\ge R\}$.
Comparison gives \eqref{eq:AF-barrier}.

The $p$-Laplace approximation of weak inverse mean curvature flow
\cite[Section~4]{MariRigoliSetti2022} (see also \cite{Moser2007})
gives locally uniform subsequential convergence of $z_\eps$
to a weak solution starting from $\Sigma$.
The limit is proper by \eqref{eq:AF-barrier} and unique by
\cite[Theorem~2.2(iii)]{HuiskenIlmanen2001}, so the whole family
converges.

Under our hypotheses, $\Sigma$ is strictly outward minimizing
as a consequence of outermostness and
\cite[Theorem~1.3(iii) and equation~(1.15)]{HuiskenIlmanen2001}.
The minimizing hull property
\cite[Property~1.4(iii)]{HuiskenIlmanen2001} and the level-set
identities in \cite[p.~369]{HuiskenIlmanen2001} then give
$z>0$ on $M\setminus\Sigma$.
Finally, locally uniform convergence yields
\[
 -\eps\log\max_{S'}u_\eps
 =\min_{S'}z_\eps
 \longrightarrow\min_{S'}z>0,
\]
which proves \eqref{eq:trace-log-limit}.
\end{proof}

With $c=\frac12\min_{\Sigma_d}z>0$,
\eqref{eq:trace-log-limit} gives, for all small $\eps$,
\begin{equation}\label{eq:trace-decay}
 a_\eps:=\max_{\Sigma_d}u_\eps\le e^{-c/\eps}.
\end{equation}
By the maximum principle, $u_\eps\le a_\eps$
on $M\setminus U$.

On $U$, the comparison principle gives
\[
 v_\eps\le u_\eps
 \le a_\eps+(1-a_\eps)v_\eps,
\]
since all three functions are $(1+\eps)$-harmonic
and their boundary values are ordered.
Together with the bound outside $U$, this yields
\begin{equation}\label{eq:global-profile}
 0\le u_\eps-\bar v_\eps
 \le a_\eps\le e^{-c/\eps}
 \qquad\text{on }M.
\end{equation}

\subsection{Exponential localization}\label{sec:localization}
Since the potential is exponentially small outside the collar, the global
problem differs from the collar problem only by exponentially small terms:
in the capacity, in the gradient, in the energy measure and in the profile.

\begin{proposition}[Exponential localization]
\label{prop:localization}
For all sufficiently small $\eps$,
\begin{equation}\label{eq:exact-collar-global}
 C_M\le C_U\le(1-a_\eps)^{-\eps}C_M,
\end{equation}
and
\begin{equation}\label{eq:global-gradient}
 \|\nabla u_\eps-\nabla\bar v_\eps\|_{L^{1+\eps}(M)}
 \le C\sqrt{a_\eps}
 \le Ce^{-c/(2\eps)}.
\end{equation}
\end{proposition}

\begin{proof}
For $0<t<1$, testing the weak equation of $u_\eps$
with $\min(u_\eps,t)-t\,u_\eps$, justified by cutoff
approximation using \eqref{eq:AF-barrier}, gives
\[
 \int_{\{u_\eps<t\}}|\nabla u_\eps|^{1+\eps}\dd V
 =t\,C_M.
\]
Since $\nabla u_\eps=0$ almost everywhere on each level set,
\begin{equation}\label{eq:energy-slicing}
 \int_{\{u_\eps>t\}}|\nabla u_\eps|^{1+\eps}\dd V
 =(1-t)\,C_M.
\end{equation}

The extension $\bar v_\eps$ is a competitor for $C_M$,
so $C_M\le C_U$. Conversely,
$(u_\eps-a_\eps)_+/(1-a_\eps)$ is a competitor for $C_U$.
By \eqref{eq:energy-slicing},
\[
 C_U
 \le (1-a_\eps)^{-1-\eps}
 \int_{\{u_\eps>a_\eps\}}
 |\nabla u_\eps|^{1+\eps}\dd V
 =(1-a_\eps)^{-\eps}C_M,
\]
which proves \eqref{eq:exact-collar-global}. In particular,
\begin{equation}\label{eq:global-comparison}
 0\le\log\frac{C_U}{C_M}
 \le-\eps\log(1-a_\eps)
 \le2\eps a_\eps
 \le2\eps e^{-c/\eps},
\end{equation}
and
\[
 0\le C_U-C_M\le2\eps a_\eps C_U.
\]

Similarly, testing the weak equation with
$\bar v_\eps-u_\eps$ by cutoff approximation gives
\[
 \int_M|\nabla u_\eps|^{\eps-1}\nabla u_\eps
 \cdot(\nabla\bar v_\eps-\nabla u_\eps)\dd V=0.
\]
Applying \eqref{eq:Lp-defect}, with all norms below
taken in $L^{1+\eps}(M)$, we obtain
\[
 \|\nabla u_\eps-\nabla\bar v_\eps\|^2
 \le
 \frac{2(C_U-C_M)}{\eps(1+\eps)}
 \bigl(\|\nabla u_\eps\|
       +\|\nabla\bar v_\eps\|\bigr)^{1-\eps}
 \le Ca_\eps.
\]
Here both gradient norms are uniformly bounded because
$C_M\le C_U$ and $C_U\to A$.
This proves \eqref{eq:global-gradient}.
\end{proof}

\subsection{Proof of Theorem~\ref{thm:global}}\label{sec:global-moments}
(i) By \eqref{eq:global-comparison}, \eqref{eq:one-sided-refined} and
\eqref{eq:kappa-area},
\[
 \log\frac{C_M}A=\log\frac{C_U}A+O(\eps e^{-c/\eps})
 =\frac\eps2|\log\eps|+\eps\Bigl(-\frac12\log\frac\pi2-\E\log g\Bigr)+\kappa_+\eps^{3/2}+O(\eps^2).
\]

(ii) Take $U=(0,d)\times\Sigma$ with the $d$ of the statement,
and extend $\mu_\eps^U$ by zero to $M$.
By \eqref{eq:power-L1} and \eqref{eq:global-gradient},
\[
 \bigl\|
 |\nabla u_\eps|^{1+\eps}
 -|\nabla\bar v_\eps|^{1+\eps}
 \bigr\|_{L^1(M)}
 \le Ce^{-c/(2\eps)}.
\]
Since $C_M,C_U\to A$ and
$C_U-C_M\le 2\eps a_\eps C_U$, we obtain
\begin{equation}\label{eq:global-measure}
 \|\mu_\eps-\mu_\eps^U\|_{\mathrm{TV}}
 \le Ce^{-c/(2\eps)}.
\end{equation}

Restriction and push-forward do not increase total variation.
Thus \eqref{eq:global-measure} and
Proposition~\ref{prop:selection-local}(ii) give
\begin{align*}
 &\bigl\|
 (R_\eps)_\#(\mu_\eps\llcorner U)
 -\nu\otimes\gamma_+(y)\dd y
 \bigr\|_{\mathrm{TV}}
 \le
 \|\mu_\eps-\mu_\eps^U\|_{\mathrm{TV}}
 +\bigl\|
 (R_\eps)_\#\mu_\eps^U
 -\nu\otimes\gamma_+(y)\dd y
 \bigr\|_{\mathrm{TV}}\le C\sqrt\eps,
\end{align*}
which proves \eqref{eq:global-TV-main}.

Finally, \eqref{eq:global-profile} and \eqref{eq:main-profile} give
\eqref{eq:global-profile-intro}, since $|u_\eps-\bar v_\eps|\le e^{-c/\eps}$.
Moreover, $v_\eps\le u_\eps\le a_\eps+(1-a_\eps)v_\eps$ on $U$ with equality
on $\Sigma$ gives
$(1-a_\eps)|\partial_tv_\eps|\le|\partial_tu_\eps|\le|\partial_tv_\eps|$ on
$\Sigma$, so \eqref{eq:main-normal-derivative} gives
\eqref{eq:global-normal-intro}.

(iii) We transfer the logarithmic moments from the collar
to $M$. Take $R$ in \eqref{eq:AF-barrier} so large that
$\{r\ge R\}\cap\overline U=\emptyset$; comparison with the
$(1+\eps)$-superharmonic function $a_\eps(R/r)^{(n-2)/\eps}$ on $\{r\ge R\}$
then gives $u_\eps\le a_\eps(R/r)^{(n-2)/\eps}$ there. By
Lemma~\ref{lem:log-control}(a), the exponential bound outside $U$, and this
bound,
\begin{equation}\label{eq:tail-envelope}
 |\nabla u_\eps|
 \le \frac{Ce^{-c/\eps}}{\eps}
 \min\Bigl\{1,\Bigl(\frac Rr\Bigr)^{(n-2)/\eps}\Bigr\}
 \qquad\text{on }M\setminus U,
\end{equation}
where the second factor is used on the AF end.

For $k=0,1,2$, the function
$s^{1+\eps}|\log(\sqrt\eps s)|^k$ is increasing
in the range of this envelope. Applying the envelope
and integrating in the AF coordinates, whose volume
element is bounded by $Cr^{n-1}\dd r\dd\omega$, gives
\begin{equation}\label{eq:global-tail}
 \int_{M\setminus U}
 |\nabla u_\eps|^{1+\eps}
 |\log(\sqrt\eps|\nabla u_\eps|)|^k\dd V
 \le Ce^{-c/(2\eps)}.
\end{equation}

On $U$, Lemmas~\ref{lem:log-control}(a) and
\ref{lem:gradient} give the required
$O(\eps^{-1})$ gradient bounds. Thus
Lemma~\ref{lem:transfer} and
\eqref{eq:global-gradient}, together with
\eqref{eq:global-tail} and
$C_M/C_U=1+O(\eps e^{-c/\eps})$, yield
\[
 \int_M
 \log^k(\sqrt\eps|\nabla u_\eps|)\dd\mu_\eps
 =
 \int_U
 \log^k(\sqrt\eps|\nabla v_\eps|)
 \dd\mu_\eps^U
 +O(\eps^{1+\eps}|\log\eps|^k)
\]
for $k=1,2$.

By \eqref{eq:main-moments} and \eqref{eq:limit-law},
\begin{equation}\label{eq:global-variance}
 \Var_{\mu_\eps}(\log|\nabla u_\eps|)
 =\frac12+\Var_\nu(\log g)
 +O(\sqrt\eps|\log\eps|^2).
\end{equation}

We finally establish the exponent differentiation formula
stated in Remark~\ref{rem:thmA}.

\begin{lemma}[Differentiation in the exponent]
\label{lem:frozen}
For $p>1$ sufficiently close to one,
\begin{equation}\label{eq:capacity-derivative}
 \frac{d}{dp}\log C_M(p)
 =\int_M\log|\nabla u_p|\dd\mu_p,
 \qquad
 d\mu_p=C_M(p)^{-1}|\nabla u_p|^p\dd V.
\end{equation}
\end{lemma}

\begin{proof}
Fix $p=1+\eps$ with $\eps>0$ sufficiently small.
Lemma~\ref{lem:log-control}(a) and
\eqref{eq:tail-envelope} imply that
\[
 \int_M|\nabla u_p|^{p+z}
 (1+|\log|\nabla u_p||^2)\dd V<\infty
\]
for all $|z|$ sufficiently small.
The weak equation and cutoff approximation also give
\[
 -\int_M X_p\cdot\nabla u\dd V=1,
 \qquad
 X_p=-\frac{|\nabla u_p|^{p-2}\nabla u_p}{C_M(p)},
\]
for every admissible competitor $u$.

Keeping $u_p$ and $X_p$ fixed while varying the
exponent to $q$, the primal and dual bounds yield
\[
\begin{aligned}
 C_M(p)
 \left(\int_M|\nabla u_p|^{\frac{p-q}{q-1}}
 \dd\mu_p\right)^{-(q-1)}
 &\le C_M(q)\\
 &\le C_M(p)
 \int_M|\nabla u_p|^{q-p}\dd\mu_p.
\end{aligned}
\]

Since $\mu_p$ is a probability measure, differentiation
under the integral sign gives
\[
 \left.\frac{d}{dz}
 \log\int_M|\nabla u_p|^z\dd\mu_p\right|_{z=0}
 =\int_M\log|\nabla u_p|\dd\mu_p.
\]
The preceding integrability also justifies the
second derivative near zero, and hence
\[
 \log\int_M|\nabla u_p|^z\dd\mu_p
 =z\int_M\log|\nabla u_p|\dd\mu_p+O_p(z^2).
\]

Now set $q=p+\eta$. Taking logarithms of the
two comparison bounds and applying this expansion,
we obtain
\[
\begin{aligned}
 &\eta\int_M\log|\nabla u_p|\dd\mu_p+O_p(\eta^2)\le \log C_M(p+\eta)-\log C_M(p)\le
 \eta\int_M\log|\nabla u_p|\dd\mu_p+O_p(\eta^2).
\end{aligned}
\]
Thus
\[
 \log C_M(p+\eta)-\log C_M(p)
 =\eta\int_M\log|\nabla u_p|\dd\mu_p
 +O_p(\eta^2),
\]
which proves \eqref{eq:capacity-derivative}.
\end{proof}

Combining \eqref{eq:capacity-derivative} with
\eqref{eq:main-first-moment} and the preceding
transfer of logarithmic moments yields
\begin{equation}\label{eq:global-Lprime}
 \frac{d}{dp}\log C_M(p)\Big|_{p=1+\eps}
 =\int_M\log|\nabla u_\eps|\dd\mu_\eps
 =\frac12|\log\eps|-\frac12\log\frac\pi2
 -\E\log g-\frac12
 +O(\sqrt\eps|\log\eps|).
\end{equation}

% ===== sec5_schwarzschild.tex =====
\section{Schwarzschild comparison and proof of Theorem~\ref{thm:sch-global}}\label{sec:boundary-gap}

Throughout, $n=3$, $K$ is the Gauss curvature of $(\Sigma,h)$,
$r_A=\sqrt{A/(4\pi)}$ and $\mathcal R=\frac12(R_{\mathfrak g}+|\mathrm{II}|^2)|_\Sigma\ge0$.
By the Gauss equation for the minimal surface $\Sigma$,
\begin{equation}\label{eq:boundary-gauss}
 V=-|\mathrm{II}|^2-\Ric_{\mathfrak g}(N,N)=K-\mathcal R,\qquad J=-\Delta_h+K-\mathcal R .
\end{equation}

The Schwarzschild exterior with horizon area $A$ satisfies the hypotheses
of Theorem~\ref{thm:global}. Its metric is scalar flat, and its horizon
is a round, totally geodesic sphere of radius $r_A$. Hence
$\mathcal R=0$, $V=K=r_A^{-2}$ and $g\equiv r_A$.
The doubled Schwarzschild metric is reflection symmetric across the
horizon, so $\kappa_+=0$ by \eqref{eq:kappa-area}.
Theorem~\ref{thm:global}(i) therefore gives
\begin{equation}\label{eq:schwarzschild-endpoint}
 \log\frac{C_{\mathrm{Sch}}(1+\eps;A)}A
 =\frac\eps2|\log\eps|
 +\eps\Bigl(-\frac12\log\frac\pi2-\log r_A\Bigr)
 +O(\eps^2).
\end{equation}

This agrees with the closed form
$C_{\mathrm{Sch}}(1+\eps;A)=4\pi r_A^{2-\eps}\,\mathrm B\bigl(\frac{2-\eps}\eps,\frac12\bigr)^{-\eps}$,
which follows from the explicit $p$-capacitary potential of the
Schwarzschild exterior \cite{XiaYinZhou2024}.

\begin{proof}[Proof of Theorem~\ref{thm:sch-global}]
Testing \eqref{eq:strict-stability} with $1$ gives
$\int_\Sigma V\dd A>0$. Since $K\ge V$, we have
$\int_\Sigma K\dd A>0$, so $\Sigma\cong S^2$ by Gauss--Bonnet.
Moreover, $-\Delta_h+K\ge J>0$, and Lemma~\ref{lem:selection}
gives a unique positive solution $G$ of
$(-\Delta_h+K)G=G^{-1}$.

Recall that $d(h)=\E\log(G/r_A)$.
We show that $d(h)\ge0$, with equality exactly when $h$ is round. Dividing
$(-\Delta_h+K)G=G^{-1}$ by $G$ and integrating with $\int_\Sigma K\dd A=4\pi$,
\[
 \E G^{-2}+\E|\nabla\log G|^2=\frac{4\pi}A=r_A^{-2}.
\]
Jensen's inequality $\E G^{-2}\ge e^{-2\E\log G}$ gives $d(h)=\E\log(G/r_A)\ge0$;
equality forces $G$ to be constant and $K=G^{-2}$, so $h$ is round, and
conversely round $h$ has $G=r_A$.

Then $g\ge G$, with equality exactly when $\mathcal R\equiv0$. Subtracting
the equations of $g$ and $G$,
$(J+(gG)^{-1})(g-G)=\mathcal R\,G\ge0$, and testing with
$\phi=\min(g-G,0)$,
\[
 0\ge\int_\Sigma\phi\,\mathcal RG\dd A
 =\int_\Sigma\Bigl(|\nabla\phi|^2+\bigl(V+(gG)^{-1}\bigr)\phi^2\Bigr)\dd A
 \ge\lambda_1(J)\int_\Sigma\phi^2\dd A,
\]
so $g\ge G$. If $\mathcal R\not\equiv0$, the strong maximum principle gives
$g>G$; if $\mathcal R\equiv0$, uniqueness gives $g=G$. Hence
\begin{equation}\label{eq:intrinsic-comparison}
 \E\log\frac g{r_A}\ \ge\ d(h)\ \ge\ 0,
\end{equation}
with equality on the left exactly when $\mathcal R\equiv0$.

\emph{The deficit.} Subtracting \eqref{eq:global-capacity-main} from
\eqref{eq:schwarzschild-endpoint},
\begin{equation}\label{eq:global-deficit-refined}
 \log\frac{C_{\mathrm{Sch}}(1+\eps;A)}{C_M(1+\eps)}
 =\eps\,\E\log\frac g{r_A}-\kappa_+\eps^{3/2}+O(\eps^2),
\end{equation}
which gives \eqref{eq:Delta-intro}; with \eqref{eq:intrinsic-comparison} this
proves $\Delta_M\ge d(h)\ge0$ and the two equality statements. In
particular, $\Delta_M=0$ exactly when $h$ is round and $\mathcal R\equiv0$.

\emph{Stability.} With $f=\log g$, dividing $Jg=g^{-1}$ by $g$ and using
\eqref{eq:boundary-gauss} and Gauss--Bonnet,
\begin{equation}\label{eq:actual-Gauss-log}
 K=g^{-2}+\Delta_hf+|\nabla f|^2+\mathcal R,\qquad
 \E g^{-2}+\E|\nabla f|^2+\E\mathcal R=\frac{4\pi}A .
\end{equation}
Put $\theta=\frac A{4\pi}(\E|\nabla f|^2+\E\mathcal R)$, so that
$\E g^{-2}=\frac{4\pi}A(1-\theta)$ and $0\le\theta<1$, and
$d\rho=g^{-2}\dd\nu/\E g^{-2}$. Since $\log\frac{4\pi}A=-2\log r_A$,
\begin{equation}\label{eq:KL-identity}
 \KL(\nu\Vert\rho)=\log\E g^{-2}+2\E\log g=2\Delta_M+\log(1-\theta).
\end{equation}
Both $\KL(\nu\Vert\rho)$ and $-\log(1-\theta)$ are nonnegative, so each is at
most $2\Delta_M$. Pinsker's inequality
$\|\nu-\rho\|_{L^1}\le\sqrt{2\KL(\nu\Vert\rho)}$, with
$\|\nu-\rho\|_{L^1}=\|g^{-2}/\E g^{-2}-1\|_{L^1(\nu)}$, gives the first bound in
\eqref{eq:Delta-stability-intro}; and $-\log(1-\theta)\le2\Delta_M$ gives
$\theta\le1-e^{-2\Delta_M}\le2\Delta_M$, which is the second.
\end{proof}

With $\lambda_1(-\Delta_h)$ the first positive eigenvalue of the Laplacian
of $(\Sigma,h)$, the Poincar\'e inequality
$\lambda_1(-\Delta_h)\Var_\nu f\le\E|\nabla f|^2$ and
Theorem~\ref{thm:global}(iii) also bound the excess variance:
\[
 \lim_{\eps\downarrow0}\Var_{\mu_\eps}(\log|\nabla u_\eps|)-\frac12=\Var_\nu(\log g)
 \le\frac{8\pi}{A\,\lambda_1(-\Delta_h)}\Delta_M .
\]

If moreover $R_{\mathfrak g}\ge0$ on $M$ and $H_2(M,\Sigma)=0$, then
$C_M(p)\le C_{\mathrm{Sch}}(p;A)$ for every $1<p<3$
\cite{XiaYinZhou2024}, and \eqref{eq:global-deficit-refined}
shows that $\kappa_+\le0$ whenever $\Delta_M=0$, that is, whenever $h$ is
round and $\mathcal R\equiv0$: in that case
$\log(C_{\mathrm{Sch}}/C_M)=-\kappa_+\eps^{3/2}+O(\eps^2)$ must be
nonnegative.

The upper capacity bound also has a purely local consequence,
requiring neither asymptotic flatness nor outermostness.
Let a smooth connected noncompact three-manifold have compact
connected orientable strictly stable minimal boundary with
$\mathcal R\ge0$, and define $C_M(p)$ by \eqref{eq:exterior-capacity}; it may vanish. The zero extension of the
trial function of Section~\ref{sec:primal} is a compactly supported
competitor, so $C_M(1+\eps)\le\Cupper$, and \eqref{eq:one-sided-refined}
for $\Cupper$, \eqref{eq:schwarzschild-endpoint} and
\eqref{eq:intrinsic-comparison} give, for small $\eps$ and a constant $B$
depending on the collar,
\[
 C_M(1+\eps)\le C_{\mathrm{Sch}}(1+\eps;A)\,
 \exp\bigl(-\eps\,\E\log(g/r_A)+B\eps^{3/2}\bigr)
 \le C_{\mathrm{Sch}}(1+\eps;A)\,e^{-\eps d(h)+B\eps^{3/2}} .
\]
The global hypotheses of Theorem~\ref{thm:sch-global} identify the boundary
coefficient $\E\log(g/r_A)$ with the limit $\Delta_M$ and $\Var_\nu(\log g)$
with the excess global fluctuation. The positivity of $d(h)$ for nonround
$h$ uses that Gauss--Bonnet fixes the mean of $K$ by the area, which is
specific to surfaces.

\appendix
% ===== appA_asymptotics.tex =====
\section{Asymptotic calculations}\label{app:asymptotics}

\subsection{A half-line Laplace lemma}\label{app:laplace}

\begin{lemma}[Half-line Laplace calculation]\label{lem:laplace}
Let $d>0$ be fixed, and suppose that, uniformly in $x$ and
in small $\eps>0$, on $0\le t\le d$,
\begin{align*}
 \phi_\eps(t,x)
 &=\tfrac12a(x)t^2+a_{3,\eps}(x)t^3-\eps B_\eps(x)t
   +O\bigl(t^4+\eps t^2+\eps^2t\bigr),\\
 q_\eps(t,x)
 &=q_{0,\eps}(x)
   \bigl(1+c_\eps(x)t+O(t^2+\eps t)\bigr).
\end{align*}
Assume that all coefficients are uniformly bounded, that $a$ and
$q_{0,\eps}$ are uniformly positive, and that
$\phi_\eps\ge ct^2-C\eps^2$ for some $c,C>0$.
Then, uniformly in $x$,
\begin{equation}\label{eq:Laplace-general}
 \int_0^d q_\eps e^{-\phi_\eps/\eps}\dd t
 =q_{0,\eps}\sqrt{\frac\eps a}
 \Biggl[
  \sqrt{\frac\pi2}
  +\sqrt\eps\left(
    \frac{B_\eps+c_\eps}{\sqrt a}
    -\frac{2a_{3,\eps}}{a^{3/2}}
  \right)
  +O(\eps)
 \Biggr].
\end{equation}
After the substitution $t=\sqrt{\eps/a}\,y$, the function
\[
 \rho_\eps(y,x)
 =\frac{q_\eps(t,x)e^{-\phi_\eps(t,x)/\eps}}
        {q_{0,\eps}(x)\sqrt{\pi/2}},
\]
extended by zero for $t>d$, satisfies, for every fixed integer $k\ge0$,
\begin{equation}\label{eq:laplace-weighted}
 \sup_x\int_0^\infty
 (1+y^k)|\rho_\eps(y,x)-\gamma_+(y)|\dd y
 \le C_k\sqrt\eps.
\end{equation}
The same bound holds after normalization of $\rho_\eps$
to total mass one.
\end{lemma}

\begin{proof}
Fix $x$ and substitute $t=\sqrt{\eps/a}\,y$. Then
\begin{align*}
 \frac{\phi_\eps}{\eps}
 &=\frac{y^2}{2}
 +\sqrt\eps\left(
 \frac{a_{3,\eps}}{a^{3/2}}y^3
 -\frac{B_\eps}{\sqrt a}y\right)
 +O\bigl(\eps(y^4+y^2)+\eps^{3/2}y\bigr),\\
 \frac{q_\eps}{q_{0,\eps}}
 &=1+\sqrt\eps\,\frac{c_\eps}{\sqrt a}y
 +O\bigl(\eps(y^2+y)\bigr).
\end{align*}
Choose a small fixed $t_0\in(0,d)$.
On $0\le t\le t_0$, the phase expansion gives
\[
 \left|\frac{\phi_\eps}{\eps}-\frac{y^2}{2}\right|
 \le C\sqrt\eps(1+y^4),
 \qquad
 \left|\frac{\phi_\eps}{\eps}-\frac{y^2}{2}\right|^2
 \le C\eps(1+y^8).
\]
Moreover, the lower bound on $\phi_\eps$ implies
\[
 e^{-\phi_\eps/\eps}+e^{-y^2/2}
 \le Ce^{-c'y^2}
\]
for some $c'>0$. Taylor's formula with integral
remainder therefore yields
\[
\begin{aligned}
 \left|e^{-\phi_\eps/\eps}
 -e^{-y^2/2}
 +e^{-y^2/2}
 \left(\frac{\phi_\eps}{\eps}-\frac{y^2}{2}\right)
 \right|
 \le
 \frac12\left|
 \frac{\phi_\eps}{\eps}-\frac{y^2}{2}
 \right|^2
 \left(e^{-\phi_\eps/\eps}+e^{-y^2/2}\right)\le C\eps(1+y^8)e^{-c'y^2}.
\end{aligned}
\]
Substituting the expansion of $\phi_\eps/\eps$
and multiplying by that of $q_\eps/q_{0,\eps}$,
we obtain
\[
\begin{aligned}
 &\left|\rho_\eps(y,x)-\gamma_+(y)
 \left[1+\sqrt\eps\left(
 \frac{B_\eps+c_\eps}{\sqrt a}y
 -\frac{a_{3,\eps}}{a^{3/2}}y^3
 \right)\right]\right|\\
 &\qquad\le C\eps(1+y^8)e^{-c'y^2}.
\end{aligned}
\]

For $t_0\le t\le d$, the assumption
$\phi_\eps\ge ct^2-C\eps^2$ makes the contribution
exponentially small, also with any fixed polynomial
weight in $y$. The Gaussian tails beyond
$t_0\sqrt{a/\eps}$ satisfy the same bound.
Thus the expansion can be integrated over $[0,\infty)$,
with an $O(\eps)$ error since
\[
 \eps\int_0^\infty(1+y^8)e^{-c'y^2}\dd y=O(\eps).
\]
Using
\[
 \int_0^\infty e^{-y^2/2}\dd y=\sqrt{\frac\pi2},
 \quad
 \int_0^\infty ye^{-y^2/2}\dd y=1,
 \quad
 \int_0^\infty y^3e^{-y^2/2}\dd y=2,
\]
and $\dd t=\sqrt{\eps/a}\,\dd y$ gives
\eqref{eq:Laplace-general}.

The same error estimate yields, for every fixed $k$,
\[
\begin{aligned}
 \int_0^\infty(1+y^k)|\rho_\eps-\gamma_+|\dd y
 &\le C_k\sqrt\eps
 \int_0^\infty(1+y^k)(y+y^3)e^{-y^2/2}\dd y\\
 &\quad+C_k\eps
 \int_0^\infty(1+y^{k+8})e^{-c'y^2}\dd y
 +O(e^{-c/\eps})\\
 &=O(\sqrt\eps),
\end{aligned}
\]
which proves \eqref{eq:laplace-weighted}.
Finally, \eqref{eq:Laplace-general} gives
$\int_0^\infty\rho_\eps\dd y=1+O(\sqrt\eps)$.
Dividing by this mass preserves the weighted
$L^1$ estimate.
\end{proof}

\subsection{Expansion of the two bounds for a general lapse}\label{app:bounds}
\begin{proof}[Proof of the expansions
\eqref{eq:primal-full} and \eqref{eq:dual-full}]
For every fixed smooth function $\varphi$ on $\Sigma$,
\[
 \frac{\int_\Sigma v^{-\eps}\varphi\dd A}
 {\int_\Sigma v^{-\eps}\dd A}
 =\E\varphi+O(\eps),
\]
and
\[
 \log\left(\frac1A\int_\Sigma v^{-\eps}\dd A\right)
 =-\eps\E\log v+O(\eps^2).
\]

\emph{The upper bound.}
By \eqref{eq:collar-expansions},
\[
 w-\eps\log(\ell/v)
 =\frac12vJv\,t^2+a_3t^3-\eps b_1t
 +O(t^4+\eps t^2).
\]
Write
\[
 \Phi_\eps(t)
 =\log\frac{\int_\Sigma\ell^{-\eps}e^w\dd A}
 {\int_\Sigma v^{-\eps}\dd A}
 =\log\frac{
 \int_\Sigma v^{-\eps}
 e^{w-\eps\log(\ell/v)}\dd A}
 {\int_\Sigma v^{-\eps}\dd A}.
\]
The oscillation in $x$ of the exponent is
$O(t^2+\eps t)$. Expanding the logarithm of
the weighted exponential average and using
the preceding weighted-average estimate, we obtain
\[
 \begin{aligned}
 \Phi_\eps(t)
 =\frac12\E(vJv)\,t^2
   +\E a_3\,t^3
   -\eps\E b_1\,t+O(t^4+\eps t^2+\eps^2t).
 \end{aligned}
\]
Strict stability gives $\E(vJv)>0$.
After decreasing $d$ if necessary,
\[
 \Phi_\eps(t)
 \ge\frac14\E(vJv)\,t^2-C\eps^2
 \qquad(0\le t\le d).
\]
Thus Lemma~\ref{lem:laplace} applies and yields
\[
 \begin{aligned}
 \int_0^d
 \left(\int_\Sigma\ell^{-\eps}e^w\dd A\right)^{-1/\eps}
 \dd t &=
 \left(\int_\Sigma v^{-\eps}\dd A\right)^{-1/\eps}
 \int_0^d e^{-\Phi_\eps(t)/\eps}\dd t\\
 &=
 \left(\int_\Sigma v^{-\eps}\dd A\right)^{-1/\eps}
 \sqrt{\frac{\pi\eps}{2\E(vJv)}}
 \bigl(1+O(\sqrt\eps)\bigr).
 \end{aligned}
\]
By \eqref{eq:primal-bound}, taking $-\eps\log$
and using the first expansion gives
\[
 \begin{aligned}
 \log\frac{\Cupper}{A}
 &=\log\left(\frac1A\int_\Sigma v^{-\eps}\dd A\right)
 -\frac\eps2\log\frac{\pi\eps}{2\E(vJv)}
 +O(\eps^{3/2})\\
 &=\frac\eps2|\log\eps|
 +\eps\left(
 \frac12\log\frac{\E(vJv)}{\pi/2}
 -\E\log v\right)
 +O(\eps^{3/2}),
 \end{aligned}
\]
which proves \eqref{eq:primal-full}.

\emph{The lower bound.}
Suppose $Jv>0$. By \eqref{eq:flux-field},
\[
 |X|=
 \frac{(Jv/v)^{\eps/2}}
 {\int_\Sigma(Jv/v)^{\eps/2}\dd A}\,e^{-w}.
\]
Since $dV=\ell e^w\dd t\dd A$, we have
\[
 \begin{aligned}
 \int_U|X|^{(1+\eps)/\eps}\dd V=
 \left(\int_\Sigma(Jv/v)^{\eps/2}\dd A
 \right)^{-(1+\eps)/\eps}
 \int_\Sigma\left(\frac{Jv}{v}\right)^{(1+\eps)/2}
 \left(\int_0^d\ell e^{-w/\eps}\dd t\right)\dd A.
 \end{aligned}
\]
Here we used
$e^{-w(1+\eps)/\eps}e^w=e^{-w/\eps}$.

The inner integral has phase and amplitude
\[
 w=\frac12vJv\,t^2+a_3t^3+O(t^4),
 \qquad
 \ell=v(1+b_1t+O(t^2)).
\]
Since $vJv>0$, Lemma~\ref{lem:laplace}
gives, uniformly in $x$,
\[
 \int_0^d\ell e^{-w/\eps}\dd t
 =v\sqrt{\frac{\pi\eps}{2vJv}}
 \bigl(1+O(\sqrt\eps)\bigr).
\]
Now
\[
 \left(\frac{Jv}{v}\right)^{1/2}
 \frac{v}{\sqrt{vJv}}=1,
\]
so the weight in the outer integral reduces to
$(Jv/v)^{\eps/2}$. Consequently,
\begin{equation}\label{eq:T-asymptotics-app}
 \begin{aligned}
 \int_U|X|^{(1+\eps)/\eps}\dd V=
 \left(\int_\Sigma(Jv/v)^{\eps/2}\dd A
 \right)^{-1/\eps}
 \sqrt{\frac{\pi\eps}{2}}
 \bigl(1+O(\sqrt\eps)\bigr).
 \end{aligned}
\end{equation}

Finally,
\[
 \log\left(
 \frac1A\int_\Sigma(Jv/v)^{\eps/2}\dd A
 \right)
 =\frac\eps2\E\log\frac{Jv}{v}+O(\eps^2).
\]
Since $\Clower$ is given by
\eqref{eq:dual-bound}, taking $-\eps\log$
in \eqref{eq:T-asymptotics-app} yields
\[
 \begin{aligned}
 \log\frac{\Clower}{A}
 &=\log\left(
 \frac1A\int_\Sigma(Jv/v)^{\eps/2}\dd A
 \right)
 -\frac\eps2\log\frac{\pi\eps}{2}
 +O(\eps^{3/2})\\
 &=\frac\eps2|\log\eps|
 +\eps\left(
 \frac12\E\log\frac{Jv}{v}
 -\frac12\log\frac\pi2
 \right)
 +O(\eps^{3/2}),
 \end{aligned}
\]
which proves \eqref{eq:dual-full}.
\end{proof}

\subsection{The two-sided layer}\label{app:two-sided}
For an interior two-sided $\Sigma$ let $U_{\mathrm{two}}=(-d,d)\times\Sigma$
with traces $1$ on $\{t=-d\}$ and $0$ on $\{t=d\}$, and let
$C_{\mathrm{two}}$ be its capacity with exponent $1+\eps$.
Let $\gamma(y)=(2\pi)^{-1/2}e^{-y^2/2}$ on $\R$ and $P(y)=\int_y^\infty\gamma$.

\begin{proposition}[Two-sided cancellation]\label{thm:two-sided}
For $v=g$,
\begin{equation}\label{eq:two-capacity}
 \log\frac{C_{\mathrm{two}}}A
 =\frac\eps2|\log\eps|+\eps\Bigl(-\frac12\log(2\pi)-\E\log g\Bigr)+O(\eps^2),
\end{equation}
and Proposition~\ref{prop:selection-local}\textup{(i)--(iii)} and Proposition~\ref{prop:moments} hold
for $U_{\mathrm{two}}$ with $y\in\R$, $\gamma$, $P$
and the limit law $-\frac12\log(2\pi)-\log g(x)-y^2/2$; the limiting
variance $\frac12+\Var_\nu(\log g)$ and the rates are unchanged.
\end{proposition}
\begin{proof}
The primal--dual constructions and the arguments of
Section~\ref{sec:local} apply with $[0,d]$ replaced by $[-d,d]$.
By \eqref{eq:g-collar} and the full-line version of
Lemma~\ref{lem:laplace}, the substitution $t=\sqrt\eps\,y$ gives
\[
\begin{aligned}
 \int_{-d}^d\ell e^{-w/\eps}\dd t
 &=g\sqrt\eps\Biggl[
 \int_\R e^{-y^2/2}
 \bigl(1+\sqrt\eps(b_1y-a_3y^3)\bigr)\dd y
 +O(\eps)\Biggr]\\
 &=g\sqrt{2\pi\eps}\,(1+O(\eps)).
\end{aligned}
\]
Indeed, the terms involving $y$ and $y^3$ are odd and
integrate to zero over $\R$.
The same cancellation holds for the primal integral
by \eqref{eq:primal-phase}, with $b_1$ and $a_3$
replaced by their $\nu_\eps$-averages.

Consequently, both resistance integrals equal
\[
 \Bigl(\int_\Sigma g^{-\eps}\dd A\Bigr)^{-1/\eps}
 \sqrt{2\pi\eps}\,(1+O(\eps)).
\]
Their ratio is $1+O(\eps)$, giving an $O(\eps^2)$
relative energy gap. Therefore,
\[
\begin{aligned}
 \log\frac{C_{\mathrm{two}}}A
 &=\log\frac{\int_\Sigma g^{-\eps}\dd A}A
   -\frac\eps2\log(2\pi\eps)+O(\eps^2)\\
 &=\frac\eps2|\log\eps|
   -\eps\Bigl(\frac12\log(2\pi)+\E\log g\Bigr)
   +O(\eps^2).
\end{aligned}
\]

The gradient selection, energy measure, profile, and logarithmic
moment estimates follow as in Section~\ref{sec:local},
using the full Gaussian in place of the half-Gaussian.
In particular, its second and fourth moments remain $1$ and $3$,
so the limiting variance is unchanged.
For the profile, the barriers of Appendix~\ref{app:barriers} are used on
$[-d,d]$: choose $K_0$ so that the drift bound there holds for $|t|\le d$,
and take the phases \eqref{eq:barriers} with
\[
 f_\pm(t)=
 \frac{\int_t^d e^{-\phi_\pm(s)/\eps}\dd s}
      {\int_{-d}^d e^{-\phi_\pm(s)/\eps}\dd s}.
\]
The same comparison gives $f_+\le v_\eps^{\mathrm{two}}\le f_-$ for the
minimizer $v_\eps^{\mathrm{two}}$ of $C_{\mathrm{two}}$, and the normalized
full-line Laplace estimate yields
$\sup_{|y|\le d/\sqrt\eps}|f_\pm(\sqrt\eps\,y)-P(y)|\le C\sqrt\eps$, which is
the uniform profile estimate, with the minimizer extended by one to the left
of the collar and zero to the right. The boundary derivatives of $f_\pm$ are
$O(\eps^{-1/2})$, so the gradient bound and the logarithmic moment argument
also apply.
\end{proof}
If the metric is reflection symmetric, uniqueness gives
$v_\eps^{\mathrm{two}}(-t,x)=1-v_\eps^{\mathrm{two}}(t,x)$ and
$C_U=2^{\eps}C_{\mathrm{two}}$ exactly, which accounts for the
difference $\log2$ of the finite coefficients; in general the two-sided
minimizer has nonconstant trace on $\Sigma$.

\subsection{One-dimensional barriers}\label{app:barriers}
Pointwise control of the collar minimizer $v_\eps$ comes from barriers
depending on $t$ only. For $f=f(t)$ with $f'<0$, the metric
\eqref{eq:collar-metric} gives
\begin{equation}\label{eq:transverse-p-laplacian}
 \Delta_{1+\eps}f
 =\ell^{-1-\eps}|f'|^{\eps-1}
 \bigl[\eps f''+
 (w_t-\eps\,\partial_t\log\ell)f'\bigr].
\end{equation}
By \eqref{eq:g-collar}, $w_t=t+O(t^2)$ and $\partial_t\log\ell=O(1)$.
Hence, after decreasing $d$, there is a constant $K_0>0$ with $K_0d<\frac14$ and
\[
 \bigl|w_t-\eps\,\partial_t\log\ell-t\bigr|\le K_0\,(t^2+\eps)
 \qquad\text{on }[0,d]\times\Sigma .
\]
Define
\begin{equation}\label{eq:barriers}
 \phi_\pm(t)=\frac12t^2
 \pm K_0\Bigl(\frac{t^3}{3}+\eps t\Bigr),
 \qquad
 f_\pm(t)=
 \frac{\int_t^d e^{-\phi_\pm(s)/\eps}\dd s}
 {\int_0^d e^{-\phi_\pm(s)/\eps}\dd s},
\end{equation}
and extend $f_\pm$ by zero for $t>d$.

\begin{lemma}[One-dimensional barriers]\label{lem:barriers}
With $f_\pm$ as in \eqref{eq:barriers}, on $\overline U$,
\begin{equation}\label{eq:barrier-sandwich}
 f_+(t)\le v_\eps(t,x)\le f_-(t),
\end{equation}
and, as $\eps\downarrow0$,
\begin{equation}\label{eq:barrier-asymptotics}
 \sup_{y\ge0}\bigl|f_\pm(\sqrt\eps\,y)-P_+(y)\bigr|\le C\sqrt\eps,
 \qquad
 \sqrt\eps\,|f_\pm'(0)|=\sqrt{\frac2\pi}\,\bigl(1+O(\sqrt\eps)\bigr).
\end{equation}
\end{lemma}

\begin{proof}
Since $\eps f_\pm''=-\phi_\pm'f_\pm'$,
\eqref{eq:transverse-p-laplacian} becomes
\[
 \Delta_{1+\eps}f_\pm
 =\ell^{-1-\eps}|f_\pm'|^{\eps-1}f_\pm'
 \bigl(w_t-\eps\,\partial_t\log\ell-\phi_\pm'\bigr),
\]
and $\phi_-'\le w_t-\eps\,\partial_t\log\ell\le\phi_+'$, while $f_\pm'<0$.
Thus $\Delta_{1+\eps}f_+\ge0$ and $\Delta_{1+\eps}f_-\le0$. The functions
$f_+$, $v_\eps$ and $f_-$ equal $1$ on $\Sigma$ and $0$ on $\Sigma_d$, so the
comparison principle gives \eqref{eq:barrier-sandwich}.

The phases $\phi_\pm$ satisfy the hypotheses of Lemma~\ref{lem:laplace}
with $a=1$, $q_\eps\equiv1$, $a_{3,\eps}=\pm K_0/3$ and $B_\eps=\mp K_0$;
the lower bound $\phi_\pm\ge\frac13t^2-3K_0^2\eps^2$ follows from $K_0d<\frac14$
and $K_0\eps t\le\frac1{12}t^2+3K_0^2\eps^2$. With $t=\sqrt\eps\,y$ we have
$f_\pm(\sqrt\eps\,y)=\int_y^\infty\hat\rho_\pm\dd\eta$, where $\hat\rho_\pm$
is the function $\rho_\eps$ of Lemma~\ref{lem:laplace} normalized to total
mass one. Since $P_+(y)=\int_y^\infty\gamma_+$, \eqref{eq:laplace-weighted}
with $k=0$ gives, for every $y\ge0$,
\[
 \bigl|f_\pm(\sqrt\eps\,y)-P_+(y)\bigr|
 \le\int_0^\infty|\hat\rho_\pm-\gamma_+|\dd\eta\le C\sqrt\eps .
\]
Finally, $|f_\pm'(0)|=\bigl(\int_0^de^{-\phi_\pm/\eps}\dd t\bigr)^{-1}$, and
\eqref{eq:Laplace-general} gives
$\int_0^de^{-\phi_\pm/\eps}\dd t=\sqrt{\pi\eps/2}\,\bigl(1+O(\sqrt\eps)\bigr)$.
\end{proof}

% ===== appB_imcf.tex =====
\section{Initial departure of the weak inverse mean curvature flow}
\label{app:imcf}

Let $z$ be the arrival time of the weak inverse mean curvature flow
starting from $\Sigma$, namely the proper weak solution of
\cite{HuiskenIlmanen2001} with initial region bounded by $\Sigma$
(Lemma~\ref{lem:log-control}(b)).
We show that, near the strictly stable minimal boundary, the arrival
time has a quadratic expansion determined by $g$.
The collar coordinates \eqref{eq:collar-metric} are built from $g$.

\begin{proposition}[Quadratic departure of the weak flow]
\label{prop:arrival-time}
Uniformly in $x\in\Sigma$,
\begin{equation}\label{eq:arrival-quadratic}
 z(t,x)=\frac12t^2+O(t^3)
 \qquad(t\downarrow0).
\end{equation}
Equivalently, in normal geodesic coordinates,
\begin{equation}\label{eq:arrival-normal}
 z(\exp_x(rN(x)))
 =\frac{r^2}{2g(x)^2}+O(r^3),
 \qquad
 g(x)=\lim_{r\downarrow0}
 \frac{r}{\sqrt{2z(\exp_x(rN(x)))}}.
\end{equation}
The limit is uniform on $\Sigma$.
\end{proposition}

\begin{proof}
We pass from the one-dimensional barriers of Lemma~\ref{lem:barriers} to
$z_\eps=-\eps\log u_\eps$. Fix $d>0$ as in that lemma. By
\eqref{eq:trace-log-limit},
\[
 a_\eps:=\max_{\{t=d\}}u_\eps\le e^{-c_d/\eps}
\]
for some $c_d>0$. Combining \eqref{eq:barrier-sandwich} with
$v_\eps\le u_\eps\le a_\eps+(1-a_\eps)v_\eps$ on $U$
(Section~\ref{sec:global}) gives
\begin{equation}\label{eq:arrival-sandwich}
 f_+(t)\le u_\eps(t,x)
 \le a_\eps+(1-a_\eps)f_-(t)
 \qquad(0\le t\le d).
\end{equation}

Fix $0<t_1<d/2$. For sufficiently small $\eps$,
$\phi_\pm'(s)\ge s/2$ on $[t_1,d]$.
Since $\phi_\pm'$ is also uniformly bounded,
for $t_1\le t\le d/2$ we have
\[
 \eps e^{-\phi_\pm(t)/\eps-C}
 \le\int_t^d e^{-\phi_\pm(s)/\eps}\dd s
 \le\frac{2\eps}{t}e^{-\phi_\pm(t)/\eps}.
\]
On the other hand, Lemma~\ref{lem:laplace} gives
\[
 \int_0^d e^{-\phi_\pm(s)/\eps}\dd s
 =\sqrt{\frac{\pi\eps}{2}}
 \bigl(1+O(\sqrt\eps)\bigr).
\]
Consequently,
\[
 -\eps\log f_\pm(t)
 =\phi_\pm(t)+O_{t_1}(\eps|\log\eps|)
 =\frac12t^2\pm\frac{K_0}3t^3
 +O_{t_1}(\eps|\log\eps|)
\]
uniformly for $t_1\le t\le d/2$.

Now \eqref{eq:arrival-sandwich} implies
\[
 -\eps\log f_+(t)
 \ge z_\eps(t,x)
 \ge-\eps\log\bigl(a_\eps+f_-(t)\bigr).
\]
Using
\[
 -\eps\log(a_\eps+f_-)
 \ge\min\{-\eps\log(2a_\eps),
          -\eps\log(2f_-)\},
\]
and the locally uniform convergence
$z_\eps\to z$ from Lemma~\ref{lem:log-control}(b),
we obtain, for every fixed $0<t<d/2$,
\[
 \min\left\{c_d,\frac12t^2-\frac{K_0}3t^3\right\}
 \le z(t,x)
 \le\frac12t^2+\frac{K_0}3t^3.
\]
For sufficiently small $t$, the second term
in the minimum is smaller than $c_d$.
This proves \eqref{eq:arrival-quadratic}, uniformly in $x$.

Finally, in normal geodesic coordinates,
$t=r/g(x)+O(r^2)$, with tangential relabeling
of order $O(r^2)$ because $\partial_t=gN$
on $\Sigma$. Substitution into
\eqref{eq:arrival-quadratic} yields
\eqref{eq:arrival-normal}.
\end{proof}

The proposition describes the initial position of the weak
flow levels. In particular, their intersections with short
outward normal rays have distance
$g(x)\sqrt{2\sigma}+O(\sigma)$ as $\sigma\downarrow0$,
without requiring the levels to be smooth normal graphs.

This differs from the capacitary boundary-layer limit.
For fixed $y>0$, the potential at $t=\sqrt\eps\,y$
converges to $P_+(y)$, uniformly in $x$ (Theorem~\ref{thm:global}(ii)), whereas
$-\log P_+(y)\ne y^2/2$.
The weak-flow limit fixes $t>0$ before taking
$\eps\downarrow0$; the boundary-layer limit keeps
$y=t/\sqrt\eps$ fixed.

\begin{remark}\label{rem:imcf-mechanism}
Proposition~\ref{prop:arrival-time} gives a geometric
interpretation of $Jg=g^{-1}$.
For a positive initial lapse $v$ with $Jv>0$,
the corresponding leaves are normal graphs of
$tv$ to first order, with mean curvature
\[
 H=t\,Jv+O(t^2).
\]
Formally, a flow with normal speed $1/H$
moves from $\{t\}$ to $\{t+\dd t\}$ in time
\[
 \dd\sigma=\ell H\dd t
 =\bigl(t\,vJv+O(t^2)\bigr)\dd t.
\]
Thus the leading arrival time is independent of
$x$ precisely when $vJv$ is constant, the same
condition that closes the Jensen gap
\eqref{eq:gap-Jensen}.
After rescaling $v$, this becomes $Jv=v^{-1}$.

For $v=g$, we obtain
$\sigma=t^2/2+O(t^3)$, in agreement with
Proposition~\ref{prop:arrival-time}.
Formally, the corresponding mean curvature is
\[
 H\approx\sqrt{2\sigma}\,Jg
 =\frac{\sqrt{2\sigma}}{g}.
\]
\end{remark}

\section*{Acknowledgements}

This research was supported by Global - Learning \& Academic
research institution for Master's$\cdot$PhD students, and Postdocs
(G-LAMP) Program of the National Research Foundation of Korea (NRF)
grant funded by the Ministry of Education
(No.~RS-2025-25442252).

Generative AI tools were used to assist with mathematical exploration, calculations, and manuscript preparation. The author reviewed and revised the resulting material.

% ===== bibliography.tex =====

\end{document}